\documentclass[10pt,reqno]{amsart}
  \usepackage[margin=1.1in]{geometry}
\usepackage{amsfonts}
\usepackage{amsmath}
\usepackage{ stmaryrd }
\usepackage{amssymb}
\usepackage{amsthm}
\usepackage{amscd}
\usepackage{ulem}
\usepackage{graphicx}
\usepackage{url}
\usepackage[all]{xy}
\usepackage{amsfonts}
 \usepackage[refpage]{nomencl}

\usepackage{enumerate,bbm}
\usepackage[english]{babel}

\def\Ck{C_{[k]}}

\def\Q{\mathbb{Q}}
\def\P{\mathbb{P}}
\def\C{\mathbb{C}}
\def\A{\mathbb{A}}
\def\Z{\mathbb{Z}}
\def\cMG{\mathcal{M}_G}

\def\l{\langle}
\def\r{\rangle}
\def\Cs{\mathbb{C}^\times}
\def\x{\times}
\def\ox{\otimes}
\def\a{\alpha}
\def\la{\lambda}
\def\ga{\gamma}

\def\pLG{L_{poly}G}

\def\tlt0{(\tilde LT/T)_0}
\def\cM{{\mathcal{M}}}

\def\mft{{\mathfrak{t}}}
\def\mfg{{\mathfrak{g}}}
\def\cG{\mathcal{G}}
\def\cO{\mathcal{O}}

\def\Gsd{L^{\ltimes}G}
\def\pGsd{L^{\ltimes}_{poly}G}

\DeclareMathOperator{\ec}{Spec }

\DeclareMathOperator{\oj}{Proj}

\newcommand{\mc}[1]{\mathcal{#1}}
\newcommand{\ol}[1]{\overline{#1}}

\makeatletter
\newcommand*\jjoin{\@ifnextchar_\jj@in\jj@@n}
\newcommand*\jj@@n{\bigvee\nolimits_{T}}
\newcommand*\jj@in[2]{\mathchoice
    {\mathop{\jj@@n}_{#2\phantom{T}}}%
    {\jj@@n_{#2}}%
    {\jj@@n_{#2}}%
    {\jj@@n_{#2}}%
}
\makeatother

\newcommand{\aba}[2]{#1 #2 {#1}^{-1}}

\makeatletter
\newtheorem*{rep@theorem}{\rep@title}
\newcommand{\newreptheorem}[2]{%
\newenvironment{rep#1}[1]{%
 \def\rep@title{#2 \ref{##1}}%
 \begin{rep@theorem}}%
 {\end{rep@theorem}}}
\makeatother

\newreptheorem{theorem}{Theorem}

\newtheorem{thm}{Theorem}[section]
\newtheorem{prop}[thm]{Proposition}
\newtheorem{lemma}[thm]{Lemma}
\newtheorem{cor}[thm]{Corollary}

\theoremstyle{definition}

\theoremstyle{remark}
\newtheorem{rmk}{Remark}
\newtheorem{ex}{Example}
 \usepackage[refpage]{nomencl}
\makenomenclature

\title{A complete degeneration of the \\ Moduli of $G$-bundles on a Curve}
\author{ Pablo Solis}
\address{Department of Mathematics,
 University of California,
 Berkeley, CA}
\email{pablo@math.berkeley.edu}

\begin{document}

\begin{abstract}
For a semisimple group $G$ it is known the moduli stack of principal $G$-bundles over a fixed nodal curve is not complete.  Finding a completion requires compactifying the group $G$.  However it was shown in \cite{Solis} that this is not sufficient to complete the moduli stack over a family of curves.  In this paper I describe how to use an embedding of the loop group $LG$ to provide a completion of the stack of $G$-bundles over a one dimensional family of curves degenerating to a nodal curve.  The completion comes with a modular interpretation inspired by the work of Gieseker,  Seshadri, Kausz and  Thaddeus and Martens.
\end{abstract}

\maketitle
\setcounter{tocdepth}{1}
\tableofcontents

\section{Introduction}
This paper introduces a moduli problem $\mc{X}_G$ of $G$-bundles on twisted curves that ``compactifies'' the moduli space of principal $G$-bundles on a family of smooth curves degenerating to a nodal curve.  More precisely, we show the moduli functor $\mc{X}_G$ satisfies the valuative criterion for completeness, which is a compactness statement for non separated spaces.

To motivate this problem we give a brief history of the subject starting with geometric invariant theory.  Fix two positive integers $r,d$. One of the first moduli problems which was intensely studied using geometric invariant theory was the moduli space $M_{r,d}(C)$ of semistable rank $r$ vector bundles of degree $d$ on a smooth curve $C$ of genus $g\ge 2$.  Mumford showed the locus of stable bundles is always a smooth quasi projective variety \cite{MR0166799, MR0175899}.  Seshadri then showed in \cite{SeshadriUnitary} that including the semistable bundles always yields a normal projective variety and hence a modular compactification when there are strictly semi stable bundles (which can happen if $(r,d) >1$).

In \cite{Ramanathan}, Ramanathan extended the notion of semistability to principal $G$-bundles; there he also constructed moduli spaces for stable $G$-bundles on a curve.  When $G$ is semisimple it was shown by Balaji, Seshadri \cite{MR1958909} and Faltings in \cite{Fa3} that there is a projective coarse moduli space $M_G(C)$ of semistable $G$-bundles providing a modular compactification of the moduli space of strictly stable bundles.  

Interest increased in these moduli spaces after a 1994 result of Faltings (for $G$ semisimple) and Beauville, Lazlo (for $G = SL_n$) regarding the global sections a particular line bundle $L$ on $M_G(C)$.  The result states that $H^0(M_G(C),L)$ coincides with the vector space of conformal blocks appearing in conformal field theory.  A crucial idea in establishing this result it to work with the moduli {\it stack} $\cMG(C)$ parametrizing all $G$-bundles on $C$.  The stack $\cMG(C)$ is not proper but is complete which means it satisfies the existence (but not uniqueness) part of the valuative criterion for properness.

The connection with conformal field theory effectively computed the dimension of $H^0(M_G(C),L)$ using a result called the Verlinde formula.  The proof of the Verlinde involves degenerating $C$ to a nodal curve where computations are easier.  The work of Faltings and Beauville,Lazlo suggested, at the very least, of considering degenerations of both $M_G(C)$ and $\cMG(C)$.

In fact the idea of degeneration had already proven useful a decade before in 1984, when Gieseker had used degeneration techniques on $M_{2,2n+1}(C)$ to prove a conjecture of Newstead and Ramanan \cite{Gieseker}.  In 1993 Caporaso used Giesker's approach to give a compactification of the moduli space of $\Cs$-bundles over the moduli space of stable curves $\ol{M}_g$.  Just a year later, Pandharipande \cite{MR1308406} gave a compactification over $\ol{M}_g$ of $M_{r,d}$ using torsion free sheaves.  In 1996, Faltings \cite{Fa1}, used torsion free sheaves to give degenerations of $M_{r,d}(C)$ and $M_G(C)$ for $G = SP_r,O_r$.  Then in the 1999 paper \cite{MR1687729}, Nagaraj and Seshadri extended Gieseker's approach to give a different degeneration for $M_{r,d}(C)$. 

One advantage of the Gieseker approach is that the resulting singularities are milder; indeed the boundary of the degeneration (the locus not parameterizing $GL_r$-bundles on the original nodal curve) is a divisor with simple normal crossings \cite[\textsection 5]{Seshadri}; in contrast the singularities for the torsion free sheaf approach are worse \cite[sect. 3]{Fa1} (they are formally smooth to the singularity at the zero matrices in the variety $\{X Y = Y X = 0\}$ with $X,Y$ square matrices).  Nagaraj and Seshadri's work seemed to solidify the Gieseker approach as a standard alternative to using torsion free sheaves.   

The remaining developments in this summary include mostly results using the Gieseker approach.  Let $\mc{M}_{r,d}$ be the moduli stack of rank $r$ vector bundles of degree $d$ and set $\mc{M}_{GL_r} = \sqcup_{d\in \Z} \mc{M}_{r,d}$. In 2005, Kausz \cite{K2} provided a degeneration of $\mc{M}_{GL_r}(C)$ using a compactification $KGL_r$ of $GL_r$.  In 2009, Tolland \cite{MR2713999}, gave a Gieseker comapactification for the moduli of $\Cs$ bundles over $\ol{M}_{g,n}$.  Recently, Martens and Thaddeus \cite{Martens} gave compactifications of arbitrary reductive groups using a Gieseker like approach to studying degenerations of $G$ bundles on genus $0$ curves.  On the other hand, Schmitt \cite{MR2127994} has provided a torsion free sheaf approach for an arbitrary semisimple group $G$ although it should be noted that the approach depends on a non canonical embedding $G \to SL(V)$.

The contribution we make here is to offer a Gieseker-like degeneration for $\mc{M}_G$ with $G$ a simple group. Let us now state the main theorem of this paper more precisely.  Let $S = \ec \C[[s]]$ and let $C_S$ be a projective curve over $S$ such that the generic fiber $C_{\C((s))}$ is smooth and the special fiber $C_0$ is a nodal curve with a single node.  Let $G$ be a connected, simple and simply connected algebraic group. We define a moduli stack $\mc{X}_G(C_S)$ parametrizing $G$-bundles on what we call twisted modifications of $C_S$.  Then $\mc{X}_G(C_S)$ contains $\cM_G(C_{\C((s))})$ as an open substack and
\begin{reptheorem}{Main Theorem}
The stack $\mc{X}_G(C_S)$ satisfied the valuative criterion for completeness: let $R = \C[[s]]$ and $K = \C((s))$; for a finite extension $K \to K'$ let $R'$ denote the integral closure of $R$ in $K'$.  Given the right commutative square below, there is finite extension $K\to K'$ and a dotted arrow making the entire diagram commute:
\[
\xymatrix{
\ec K' \ar[r]\ar[d] & \ec K\ar[d]\ar[r]^{h^*} & \mc{X}_G(C_S)\ar[d]\\
\ec R' \ar@{-->}[urr]^{ \ \ \ \ \ \ \ \ \ \ \  \ \ \ \ \ \ \ \ \ \ \ \ \ h}\ar[r] & \ec R\ar[r]^f & S
}
\]
\end{reptheorem}

The approach of this paper is to use the connection between loop groups and the moduli of principal bundles on curves as well as a recently defined embedding of the loop group \cite{Solis}. Further, because we work with stacks, this approach works in all genus and works for both reducible and irreducible nodal curves.

We now elaborate on the notion of a twisted modification. Specifically a twisted modification $\mc{C}'_S$ of $C_S$ is a curve over $S$ with a map $\mc{C}' \xrightarrow{f} C$ such that if $C^*_S = C_S \backslash \{p\}$ with $p$ the node, then $f^{-1}(C^*_S) \to C^*_S$ is an isomorphism and $f^{-1}(p)$ is $[R_n/\mu_k]$ where $R_n$ is a connected chain of $\P^1$s (see figure \ref{chain3}), $\mu_k$ is the group of $k$th roots of unity and the value of $k$ is determined by $G$.
\begin{figure}[htm]
\centering
\begin{picture}(90,20)
\thicklines
\line(3,1){30}
\put(-15,10){\line(3,-1){40}}
\put(10,-5){\line(3,1){35}}	
\end{picture}
\caption{A chain of $\P^1$s of length $3$.}
\label{chain3}
\end{figure}
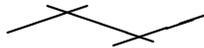  
The stack $\mc{X}_G(C_S)$ parametrizes $G$-bundles on twisted modifications $\mc{C}'$ of $C_S$ where the $G$-bundle has prescribed equivariant structure on the fixed points of the $\mu_k$ action on $f^{-1}(p)$.  We call such an object a twisted Gieseker bundle on $C_S$.

The restriction of the $G$ bundle to the chain $[R_n/\mu_k]$ is a $\mu_k$ equivariant $G$-bundle.  The use of equivariant $G$-bundles on chains is an idea introduced by Martens and Thaddeus in \cite{Martens}.  In fact they worked with $\Cs$-equivariant bundles but both Martens and Thaddeus had mentioned to me that they considered working with $\mu_k$-equivariants and that it could be a viable alternative.  

At the same time I was lead to consider $\mu_k$ equivariant $G$-bundles for an entirely different reason.  Namely, under the base change $S \xrightarrow{s \mapsto s^k} S$, the standard genus $0$ degeneration to a node $\C[x,y,s]/(xy - s)$ becomes $\C[x,y,s]/(xy - s^k)$ which can be identified with $\mu_k$ invariants in $\C[u,v]$ where $\zeta \in \mu_k$ acts by $\zeta(u,v) = (\zeta u, \zeta^{-1} v)$.  This observation had been made and used by both Faltings and Seshadri.  The step taken here was to combine this observation with equivariant bundles on chains to arrive at the definition of $\mc{X}_G(C_S)$. Finally, I relate the geometry of $\mc{X}_G(C_S)$ to the geometry of the loop group embedding constructed in \cite{Solis} to show the valuative criterion of completeness for $\mc{X}_G(C_S)$.

The basic idea for the proof of theorem \ref{Main Theorem} is as follows.  Working in a neighborhood of the node, the moduli space of $G$-bundles on these equivariant chains is naturally isomorphic to a certain orbit in the embedding $\ol{\Cs\ltimes \pLG}$ of the polynomial loop group from \cite{Solis}.  This allows one to show that the objects in $\mc{X}_G(C_S)$ degenerate in way that corresponds to a $\pLG$-orbit stratification of $\ol{\Cs\ltimes \pLG}$ and consequently deduce the completeness statement form a corresponding completeness statement for $\ol{\Cs\ltimes \pLG}$.

The outline of the paper is as follows.  Section 3 contains a discussion of some of the subtler points about the moduli spaces $M_G$ and $\cMG$.  It also contains some standard arguments used throughout the paper.  Section 4 develops results on $G$ bundles on twisted curves.  When $C$ is a fixed smooth curve there is some overlap with \cite{Balaji}.  We then proceed to a fixed nodal curve with a single node, and then to fixed curve where the node has been replaced with a $\mu_k$ equivariant chain.  In section 5 we define precisely the moduli problem $\mc{X}_G(C_S)$ and prove the main theorem.
\bigskip

\noindent {\bf Acknowledgements}
I would like to thank my advisor Constantin Teleman for numerous helpful suggestions and constant encouragement throughout this project.  I thank Michael Thaddeus for explaining his ideas about equivariant bundles and Johan Martens for many helpful discussions.  I would like to thank V Balaji and C.S. Seshadri for explaining their work parahoric torsors and pointing me to the previous work \cite{MR1687729,Seshadri} of Nagaraj and Seshadri.

\section{Basic constructions, conventions and notation}
Here we pin down conventions for various tools, construction and other notation used throughout the paper.  This is an attempt to delegate notation building here and have the other sections focused on proving the main theorem.

\subsection{Groups and Lie algebras}\label{s:notation.Lie} We use $G$ to denote a simple, connected and simply connected algebraic group $G$ over $\C$ and $T \subset G$ a maximal torus. Let $\mfg = Lie(G)$, $\mft = Lie(T)$ and let $\Delta \subset \mft^*$ be the roots so that $\mfg = \mft \oplus_{\a \in \Delta} \mfg_\a$. Let $\Delta^+$ be a choice of positive roots so that $\Delta = \Delta^+ \cup - \Delta^+$. Let $r = \dim T$ and $\a_1, \dotsc, \a_r$ denote an ordered choice of simple roots.

We have a parallel set of conventions for the loop group $LG$.  As a functor, the loop groups is defined on $\C$-algebras via $LG(R):= G(R((z)))$. Similarly, the polynomial loop group is $\pLG(R):= G(R[z^\pm])$.

There is a strong parallel between $LG$ and $G$ which is best seen by introducing $\Gsd := \Cs \ltimes LG$ or $\pGsd := \Cs \ltimes \pLG$. The group structure is given by 
\begin{align*}
(u_1,\ga_1(z))\cdot (u_2, \ga_2(z)) &= (u_1u_2, u_2^{-1}\ga_1(z)u_2 \ga_2(z))\\
u_2^{-1}\ga_1(z)u_2 &= \ga_1(u_2^{-1} z)
\end{align*}

A maximal torus for $\Gsd$ is $\Cs \x T$ for any maximal torus $T\subset G$. In sections 4,5 we work with $\pLG$ and it's Lie algebra $Lie(\pLG) = \mfg\ox \C[z^\pm] =: \mfg[z^\pm]$. Define $d$ by $ Lie(\Cs \x T) = \C d\oplus \mft$. We are now in a position to set up analogous root notation for $\mfg[z^\pm]$ and it is conventional to use the term affine to differentiate it from the notation for $\mfg$. The root spaces for $\mfg[z^\pm]$ are of the form $z^i \mfg_\a$ and $z^j \mft$. Let $\Delta^{aff} \subset (\C d \oplus \mft)^*$ be the subset so that 
\[
\C d \oplus \mfg[z^\pm] = \C d \oplus \mft \bigoplus_{(n,\a) \in \Delta^{aff}} z^n\mfg_\a. 
\]
Then the elements of $\Delta^{aff}$ are called the {\it affine roots}. Let $z^i \Delta$ stand for the roots of the form $(i, \a)$ for $\a \in \Delta$. A choice of positive roots is $\Delta^{aff,+} = \Delta^+ \bigcup_{i \ge 1} z^i \Delta \cup \{(i, 0)\}$.

Let $\theta$ denote the longest root in $\mfg$.  The simple roots for $\Cs \ltimes \pLG$ are $(0,\a_1), \dotsc, (0, \a_r), (1, - \theta)$. All of this notation also applies to $\Gsd = \Cs \ltimes LG$.  By abuse of notation we denote $(0,\a_i)$ with $\a_i$ and set $\a_0 = (1,-\theta)$.

\subsubsection{(Co-)Characters, Parabolic  and Parahoric Subgroups}\label{s:notation.para}
For any torus $T$ we have the lattice of characters $\hom(T,\Cs)$ and co-characters $\hom(\Cs,T)$. Further, for $(\eta, \chi) \in \hom(\Cs,T) \x \hom(T,\Cs)$ we set $\l \eta, \chi\r := \chi\circ \eta \in \Z$.  

For $T \subset G$ a maximal torus and for $\eta \in \hom(\Cs,T)$ the set $P(\eta):= \{g \in G| \lim_{t\to 0} \aba{\eta(t)}{g} \ \mbox{exists}\}$ is a subgroup.  A {\it parabolic} subgroup is any subgroup $P \subset G$ conjugate to some $P(\eta)$.  

We can apply the same construction for $\eta \in \hom(\Cs, \Cs \x T)$ to get a subgroup $P(\eta) \subset \Gsd$.  A {\it parahoric} subgroup is any group conjugate to one of the $P(\eta)$.  By abuse of notation, we use $P(\eta)$ to denote its image under the projection $\Gsd \to LG$.  Parahoric subgroups of $LG$ are any subgroups conjugate to one of the $P(\eta)$. 

Parabolic and parahoric subgroups come with natural factorizations $P(\eta) = L(\eta) U(\eta)$ known as a Levi decomposition: $L(\eta) = \{g \in G| \lim_{t\to 0} \aba{\eta(t)}{g}  = g\} $ and $U(\eta) = \{g \in G| \lim_{t\to 0} \aba{\eta(t)}{g} = 1\}$.  A simple is example comes from $\eta_0 \colon \Cs \to \Cs \x T$ defined by $\eta_0(t) = (t, 1)$.  Then $\aba{\eta_0(t)}{g(z)} = g(t z)$ and $P(\eta_0) = G[[z]] = G(\C[[z]]) =: L^+G$.  The Levi factorization is $G \cdot N$ where $N$ is the kernel of the map $G[[z]] \xrightarrow{z \mapsto 0} G$.

By $\mft_\Q$ we denote $\hom(\Cs,T)\ox_\Z \Q$.  The Weyl chamber is defined as $Ch:= \{\eta \in \mft_\Q| \l \a_i, \eta \r \ge 0\}$. It is a simplicial cone whose faces are given by $\{ \l \a, \eta \r = 0 | \a \in I\}$ for subsets of $I \subset \{\a_1, \dotsc, \a_r \}$.   

Similarly, we have the affine Weyl chamber $Ch^{aff} = \{\eta \in \Q\oplus\mft_\Q | \l \a_i, \eta \r > 0\}$; now the faces are in bijection with subsets $\{\a_0, \dotsc, \a_r\}$. It is convention to instead work with the affine Weyl alcove $Al:= Ch^{aff} \cap 1 \oplus \mft_\Q =  \{\eta \in \mft_\Q| 0\le \l \a_i, \eta \r, \l \theta, \eta \r \le 1 \}$.  A {\it face} $F$ of $Al$ is $F' \cap 1 \oplus \mft_\Q$ where $F'$ is a face of $Ch^{aff}$. 

Any $\eta \in Ch$ determines a fractional co-character $\Cs \to T$ but nevertheless a well defined parabolic $P(\eta)$.  Any parabolic is conjugate to some $P(\eta)$ and if $\eta,\eta'$ are in the interior of the same face then $P(\eta) = P(\eta')$.  Similarly any $\eta \in Al$ determines a parahoric $P(\eta) \subset LG$. Any parahoric is conjugate either to $P(\eta)$ or to $P(-\eta)$. Let $Al_e = \{\eta \in Al | \l \theta, \eta \r = 1\}$. If $\eta \in Al_e$ the resulting parahoric is called {\it exotic}. Alternatively,  the inclusion $\{\a_1, \dotsc, \a_r\} \subset \{\a_0, \dotsc, \a_r\}$ defines a map from faces of $Ch$ to those of $Al$.  The faces missed by $Ch$ are exactly those contained in $Al_e$.  

The exotic parahorics give rise to moduli spaces of torsors on curves which are not isomorphic with moduli spaces of $G$-bundles.  Informally then the exotic parahorics can be viewed as geometry only visible to $LG$.  Exotic parahorics are studied in depth in \cite{Balaji}; there they are called nonhyperspecial maximal parahoric subgroups.

The ordered simple roots $\{\a_0, \a_1, \dotsc, \a_r\}$ determine ordered vertices $\{\eta_0,\dotsc, \eta_r\}$ determined by the conditions $\l \eta_i, \a_j \r = 0$ for $i \ne j$ and $\l \eta_0, \a_0 \r = 1$.  If we write $\theta = \sum_{i=1}^r n_i \a_i$ and set $n_0 = 1$ then one can check these condition can be expressed as
\begin{equation}\label{para.eta}
\begin{aligned}
  & \langle \alpha_i, \eta_j \rangle = \frac{1}{n_i}\delta_{i,j} & 
\end{aligned}
\end{equation}

Now for each $I \subset \{0, \dotsc, r\}$ we define $\eta_I = \sum_{i \in I}\eta_i$.  Then $\eta_I$ lies in the face of $Al$ associated to the complement of $I$; if $I = \emptyset$ we take $\eta_I$ to be the trivial co-character. Finally, we set
\begin{equation}\label{para.etaI}
\begin{aligned}
\mc{P}_I = P(\eta_I) && \mc{P}^-_I = P(-\eta_I)\\
\mc{U}_I = U(\eta_I) && \mc{U}_I^- =  U(-\eta_I)\\
 &L_I = L(\eta_I) = L(-\eta_I)&
\end{aligned}
\end{equation}
One can check that $\mc{P}_I = \cap_{i \in I} P(\eta_i)$. It is sufficient to establish this at the level of Lie algebras because $P(\eta)$ is connected $\forall \eta$ ( the map $g \mapsto \lim_{t \to 0} \eta(t) g \eta(t)^{-1}$ defines a retraction onto the Levi factor which is connected).  Returning to Lie algebras,  we note the $\supset$ direction is routine to verify.  Going the other way we have $Lie(\mc{P}_I)$ is spanned by $\C d \oplus \mft$ and those $X_\a$ for which $\l \a, \eta_I \r \ge 0$. It is suffices to work with $\a$ negative so that $0 \ge \l \a, \eta_i \r \forall i$.   Then we have $0 \ge \sum_{i \in I} \l \a , \eta_i \r =  \l \a, \eta_I \r \ge 0$ which is only possible if each term is equal to $0$; i.e. $X_\a \in Lie(\mc{P}_i) \forall i \in I$.

\subsection{(Equivariant)-Bundles, Quotient Stacks and torsors}\label{s:notation.root}
Let $H$ be a linear algebraic group over $\C$. A principal $H$-bundle over a base scheme $B$ is scheme $P$ with a smooth map $P \to B$ such that any $p \in B$ has an fppf neighborhood $B'$ such that $P\x_B B' \cong B' \x H$. Because all of our group schemes are smooth we can equivalently require local triviality in the \'etale topology but below we generally work on curves with fppf covers coming from formal neighborhoods of points.

Given a scheme $B$ equipped with an action of an algebraic group $H$ we can form the quotient stack $[B/H]$. By definition a morphism $B' \to [B/H]$ is the data of $\empty$a principal $H$-bundle $P$ over $B'$ together with an $H$-equvariant map $P \to B$. Quotient stacks play a prominent role in our use of twisted curves defined in the next section.

Given a base $B$ with the action of a group $\Pi$ an equivariant $H$-bundle on $B$ is a bundle $P\to B$ together with an action of $\Pi$ making the following diagram commute
\begin{equation}\label{eq:equivariant.structure}
\xymatrix{
\Pi \x P \ar[d]\ar[r] & P\ar[d]\\
\Pi \x B \ar[r]&  B}
\end{equation}
Equivalently, or by definition, an equivariant $H$-bundle is a $H$-bundle on $[B/\Pi]$.  For $b \in B$ let $\Pi_b$ denote the stabilizer of $b$ in $\Pi$.  Then the above diagram produces an action of $\Pi_b$ on the fiber of $P$ over $b$.  The action is determined by a representation $\rho \colon \Pi_b \to H$.  In general we summarize this situation by saying that {\it the equivariant structure of $P$ at $b$ is given by $\rho$.}

Let $G$ be a connected, simply connected simple group over $\C$. The basic source of equivariant bundles in this paper are $G$-bundle on $[\ec \C[[z]]/\mu_k]$ where $\zeta \in \mu_k$ acts by $z \mapsto \zeta z$. Any $G$-bundle on $\ec \C[[z]]$ is trivial and so an equivariant bundle is determined by its equivariant structure $\mu_k \to G$ at the closed point of $\ec \C[[z]]$.

We also utilize torsors for a sheaf of groups $\mc{G}$. In general, given a curve $C$ and a sheaf of groups $\cG$ on $C$ we define a {\it $\cG$-torsor} to be a sheaf of sets $\mathcal{F}$ on $C$ together with a right action of $\cG$ such that (1) there is a fppf cover $\{C_i \to C\}$ such that $\mathcal{F}(C_i) \ne \emptyset$ and (2) the action map $\cG \x \mathcal{F} \to \mathcal{F} \x \mathcal{F}$ is an isomorphism.

Given $G$ as above, we can form the sheaf $U \mapsto \hom_{sch}(U,G)=: \mc{G}^{std}(U)$.  Generally our sheafs of groups agree with $\mc{G}^{std}$ on an open set $U \subset C$ but in general have more intricate behavior $C\backslash U$.  Torsors for $\mc{G}^{std}$ can be identified with $G$ bundles and so the notion is most relevant when working with a sheaf of groups $\mc{G} \ne \mc{G}^{std}$; we often write simply {\it torsor} to indicate a torsor for a sheaf of groups $\mc{G} \ne \mc{G}^{std}$ to be specified later. Examples of torsors are given in \ref{s:background.torsors}.

\subsection{Conventions on Curves}\label{s:notation.curve}
Generally we work over $\ec \C$ and a scheme will mean a scheme over $\ec \C$.  Let $S$ be a scheme. We denote a flat family of curves $C \to S$ as $C_S$. If $B$ is an $S$-scheme then $C_B:= C_S \x_S B$.  For affine schemes $\ec R \to S$ we write $C_R$ for $C_{\ec R}$. 

Generally we work with a fixed curve over $\ec \C$ or with a family of curves over $S = \ec \C[[s]]$.  Set $S^* = \ec \C((s))$ and $S_0 = \ec \C = \ec \C[[s]]/(s)$ the closed point.  Then $C_S$ always denotes a curve with generic fiber $C_{S^*}$ smooth and special fiber $C_0 := C_{S_0}$ nodal with unique node $p$.  We write $C_S - p$ for the open subscheme $C_S \backslash \{p\}$.  We also assume $C_S$ is a regular surface as scheme over $\ec \C$.

For any closed point $p$ in a scheme $Z$ we denote by $\hat \cO_{Z,p}$ the completion of $\cO_{Z,p}$ with respect to the maximal ideal.  We often use $D$ to denote a formal neighborhood of a point in a curve.  The cases that will arise are
\begin{itemize}
\item $p \in C$ a smooth curve, $\hat \cO_{C,p} \cong \C[[z]]$ and we set $D = \ec \C[[z]]$
\item $p \in C_0$ is the node, $\hat \cO_{C,p}\cong  \C[[x,y]]/xy$ and we set $D_0 = \ec \C[[x,y]]/(xy)$
\item $p \in C_S$ is the node, $ \hat \cO_{C_S,p} \cong \frac{\C[[s,x,y]]}{(xy - s)} \cong \C[[x,y]]$ and we set $D_S = \ec \C[[x,y]]$
\item for $k \ge 2$ and $k$th roots $u,v$ of $x,y$ we set $D_S^{\frac{1}{k}} = \ec \C[[u,v]]$
\end{itemize}

The last case arises as follows.  We first notice that if we base change $D_S$ under $s \mapsto s^k$ then $D_S$ becomes $\ec \C[[x,y,s]]/(xy - s^k)$.  If we let $\mu_k$ denote the $k$th roots of unity then $\ec \C[[x,y,s]]/(xy - s^k) = D_S^{\frac{1}{k}}//{\mu_k}$ where $\zeta \in \mu_k$ acts by $\zeta (u,v) = (\zeta u, \zeta^{-1} v)$Êfor $\zeta \in \mu_k$.  A basic strategy we employ is to replace the curve $\ec \C[[x,y,s]]/(xy - s^k)$ with the orbifold or twisted curve $[D_S^{\frac{1}{k}}/\mu_k]$.

In section 5 we utilize results on twisted curves from \cite{MR2786662}. We now recall the definition of a twisted curve (with no marked points) in characteristic 0. A twisted nodal curve $\mc{C} \to S$ is a proper Deligne-Mumford stack such that
\begin{itemize}
\item[(i)] The geometric fibers of $\mc{C} \to S$ are connected of dimension $1$ and such that the coarse moduli space $C$ of $\mc{C}$ is a nodal curve over $S$.
\item[(ii)] If $\mc{U} \subset \mc{C}$ denotes the complement of the singular locus of $\mc{C} \to S$ then $\mc{U} \to C$ is an open immersion.
\item[(iii)] Let $p \colon \ec k \to C$ be a geometric point mapping to a node and let $s \in S$ denote the image of $\ec k$ under $C \to S$ and let $m_{S,s}$ denote the maximal ideal of the local ring $\mc{O}_{S,s}$. Then there is an integer $k$ and an element $t \in m_{S,s}$ such that 
\[
\ec \cO_{C,p} \x_C \mc{C} \cong [D^{sh}/\mu_k]
\]
where $D^{sh}$ denotes the strict henselization of $D := \ec \cO_{S,s}[u,v]/(u v - t)$ at the point $(m_{S,s}, u, v)$ and $\zeta \in \mu_k$ acts by $\zeta (u,v) \mapsto (\zeta u, \zeta^{-1} v)$. 
\end{itemize}

We did not mention markings because largely we will not make use of them except for one exception. If $C$ is a smooth curve we can twist at a marked point $p$ as described below.  Let $p \in C$ and $D = \ec \C[[z]]$ as in the first bullet point above and fix a positive integer $k$ and a $k$th root $w$ of $z$.  We have $\ec \C((w))/\mu_k = \ec \C((z))$ so let $C_{[k]}$ denote $C - p \cup_{\ec \C((z))} [\ec \C[[w]]/\mu_k]$. It is a twisted curve whose coarse moduli space is $C$.

In a similar fashion, with $C_0, C_S$ as in the bullet points, we can construct twisted curves $C_{0,[k]}$ and $C_{S,[k]}$ with coarse moduli space $C_0, C_S$ and such the the fiber of the node is $[pt/\mu_k]$.

\section{Survey of Facts about $\mc{M}_G(C)$}
The problem of compactifying $G$-bundle on nodal curves involves some subtleties that are well known to the experts but are nevertheless worth stating explicitly. These subtleties include coarse moduli spaces vs stacks, issues on nodal curves, Gieseker bundles vs torsion free sheaves, and the connection with the loop group.

\subsection{$\cMG$, $M_G$, completeness and compactness}
Let $H$ be reductive group over $\C$. If $C$ is a smooth curve of genus $g$ over $\ec \C$ then there is a stack $\mc{M}_H(C)$ parametrizing principal $H$-bundles on $C$. It is a smooth algebraic stack of dimension $\dim H (g-1)$. Further there is a universal bundle $P^{univ} \to C \x \mc{M}_H(C)$ such that if $P\to C\x B$ is any $H$-bundle then there is a morphism $B \xrightarrow{f} \mc{M}_H(C)$ such that $P \cong (id,f)^*P^{univ}$.

Let us now specialize to groups $G$ as in \ref{s:notation.Lie}. It is known that $Pic(\mc{M}_G(C)) = \Z$ and there is a generator $L$ which is ample. Using $L$, one constructs the coarse moduli space of semistable $G$-bundles $M_G(C) = \oj \bigoplus_n \Gamma(\mc{M}_G(C),L^{\ox n})$ \cite[\textsection 8]{T2}. This is not the conventional construction but illustrates how $M_G(C)$ can be recovered from $\mc{M}_G(C)$. On the other hand, $M_G(C)$ has the advantage of being a projective variety and hence compact whereas $\mc{M}_G(C)$ is not separated and thus not compact. 

The case of $\mc{M}_{SL_2}(\P^1)$ is an instructive example. As a set, $\mc{M}_{SL_2}(\P^1) = \mathbb{N}$ where $n$ corresponds to the bundle $\mc{O}(n)\oplus \mc{O}(-n)$ where we abbreviate $\mc{O} = \mc{O}_{\P^1}$. Further, the ample generator $L \in Pic(\mc{M}_{SL_2}(\P^1))$ satisfies $H^0(\mc{M}_{SL_2}(\P^1),L^{\ox n}) = \C$ so $M_{SL_2}(\P^1) = \oj \C[t] =\ec \C$ which corresponds to $\mc{O}\oplus \mc{O}$, the unique semistable bundle.

Further there is a vector bundle $E\to \P^1 \x Ext^1(\mc{O}(1),\mc{O}(-1))$ \cite[Lemma 3.1]{NarasimhanSeshadri} such that $E|_{\P^1 \x v}$ corresponds to the extension $v \in Ext^1(\mc{O}(1),\mc{O}(-1)) = H^1(\mc{O}(-2))= \C$. For $v \ne 0$ this extension is the Euler sequence
\[
0\to \mc{O}(-1) \to \mc{O}\oplus \mc{O} \to \mc{O}(-1) \to 0
\]
Comparing with the trivial family $p_1^* (\mc{O}\oplus \mc{O})$ on  $ \P^1 \x \A^1$ we get two maps $\A^1 \xrightarrow{f_1,f_2} \mc{M}_{SL_2}(\P^1)$ that agree on $\Cs$ such that $f_1(0) = \mc{O}\oplus \mc{O}$ and $f_2(0) =  \mc{O}(1)\oplus \mc{O}(-1)$. This shows $\mc{M}_{SL_2}(\P^1)$ is not separated and further $0,1 \in \mc{M}_{SL_2}(\P^1)$ are in the same connected component; this construction generalizes to show $\mc{M}_{SL_2}(\P^1)$ is connected.  More generally, $\pi_0(\mc{M}_G(C)) = \pi_1(G)$.

Because of this behavior, we can at most ask for $\mc{M}_G(C)$ to satisfy the existence part of the valuative criterion for properness; this is called completeness. Specifically, a morphism of stacks $X \to Y$ is complete if for every complete discrete valuation ring $R$ with fraction field $K$ and every diagram with solid arrows there exists a dotted arrow making the diagram commute.
\[
\xymatrix{
\ec K' \ar[r]\ar[d] & \ec K\ar[d]\ar[r] & X\ar[d]\\
\ec R' \ar@{-->}[urr]\ar[r] & \ec R\ar[r] & Y
}
\]
where $K \to K'$ is a finite extension $R'$ is the integral closure of $R$ in $K'$.

If $C$ is a smooth curve, then $\mc{M}_G(C) \to \ec \C$ is complete. Completeness fails when $C$ is nodal as is discussed in the next section.

\subsection{Nodal Curves and the case of $GL_n$}
If $C$ is a nodal curve then $\cMG(C)$ may be complete. For the group $\Cs$ this holds on any curve of compact type.  If $C$ is a chain of $\P^1$s and $H$ is reductive then $\mc{M}_H(C)$ is discrete and naturally isomorphic to $\mc{M}_H(\tilde C)$ hence complete \cite[Variation 4]{MartensThaddeus}.  But as soon as the irreducible components of $C$ have genus $\ge 2$ then $\mc{M}_H(C)$ will not be complete. Even if the genus is $1$ we will run into trouble as the next example shows. 

Consider the curve  $C = \{y^2 - x^2(x+1) = 0\} \subset \A^2$. Consider the divisor defined on $C \x \Cs$ defined by the section $t \mapsto (t^2 + 2t , (t^2+2t)(t+1) , t)$. This defines a line bundle on $C \x \Cs$. The limit as $t\mapsto 0$ is the nodal point which doesn't define a line bundle but a rather a torsion free sheaf.  By enlarging the moduli problem to parametrize torsion free sheaves one can get a compact coarse moduli space \cite{Fa1,Seshadri}.

A key insight originally due to Gieseker \cite{Gieseker} is that torsion free sheaves can be replaced by vector bundles on modified curves. Specifically, on any nodal curve $C$ with nodes $\{p_1, \dotsc, p_m\}$, a torsion free sheaf $\mc{F}$ on $C$ can be realized as the pushforward of a vector bundle $F$ on a modification $C' \xrightarrow{\pi} C$ where $\pi^{-1}(C - \{p_1, \dotsc, p_m\} ) \to C$ is an isomorphism and $\pi^{-1}(p_i)$ is a chain of projective lines of length at most the rank of $\mc{F}$.  Further, if $\pi_*(F) = \mc{F}$ then for each $\P^1 \subset \pi^{-1}(p_i)$ it is necessary that $F|_{\P^1} = \cO(1)^{\oplus i}\oplus \cO^{rk(\mc{F}) - i}$ with $i > 1$ and that $H^0(\pi^{-1}(p_i), F|_{\pi^{-1}(p_i)} \ox \cO(-p_i'-p_i'')) = 0$ where $p_i',p_i''$ denote the extreme points on the chain $\pi^{-1}(p_i)$.

\subsection{Torsors versus $G$ bundles}\label{s:background.torsors}
See \ref{s:notation.root} for the definition of a torsor for a sheaf of groups $\mc{G}$.  The point of discussing $\cG$-torors is that a family of $G$ bundles over a nodal curve can limit to a $\cG$ torsor which cannot be identified with a $G$-bundle.

Starting with $G$ we can form the sheaf of groups $\cG^{std}(U):= \hom_{Sch}(U,G)$. Any principal bundle $F$ on $C$ defines a torsor $\mc{F}$ for $\cG^{std}$ by $\mc{F}(U) \mapsto Sect(U,F|_U)$.  In fact in much the same way vector bundles can be identified with locally free sheaves, $G$ bundles can be identified with $\cG^{std}$-torsors.

\nomenclature{$\cG^{std}$}{Sheaf of groups associated to the constant group scheme.}

More generally let $P\subset G$ be a parabolic subgroup. Let $L^+_{P}G = \{\ga \in G[[z]] \ | \ga(0) \in P\}$. Construct a sheaf of groups $\mc{G}^P$ on $\ec \C[[z]] = \{(z), (0)\}$ by $\mc{G}^P(\{(z), (0)\}) = L^+_{P}G$ and $\mc{G}^P(\{0\}) = G((z))$. Given a smooth curve $C$ and a point $p$ we notice that $\cG^{std}|_{C-p}$ and $\mc{G}^P$ agree over $\ec \C((z)) \cong C-p \x_C \ec \hat \cO_p$ and thus define a sheaf of group which we also denote $\mc{G}^P$.  Clearly we can iterate over $(x_i) = x_1,\dotsc, x_m \in C$ with parabolics $(P_i) = P_1, \dotsc, P_m$.  Call the resulting sheaf of groups $\cG^{(x_i),(P_i)}$.  Then $\cG^{(x_i),(P_i)}$-torsors are exactly quasi parabolic bundles: $G$-bundles on $C$ with reduction of structure group to $P_i$ at $x_i$.
\nomenclature{$L^+_{P}G$}{Parahoric subgroup of $LG$ associated to a parabolic of $G$}

In the examples mentioned thus far all the $\cG$-torsors can be identified with $G$-bundles potentially with additional structure; this is not always the case.  The groups $L^{+}_PG$ are parahoric subgroups and we can apply the same construction to any parahoric subgroup $\mc{P}$ (see \ref{s:notation.para} in particular for the definition of (exotic) parahorics).  Specifically, given a set $(\mc{P}_i)$ of parahoric subgroups we can analogously construct a sheaf of groups $\cG^{(x_i),(\mc{P}_i)}$. When the parahorics are exotic the resulting moduli spaces are not isomorphic to moduli spaces of $G$-bundles on $C$; see remark \ref{rmk:exotic} after corollary \ref{c:para2equiv.smooth}.
\nomenclature{$\cG(\hat \cO_x)$}{completed stalk of a sheaf of groups}

\subsection{The double coset construction}\label{s:background.DCC}
There is a close connection between the loop group $LG$ and the moduli stack $\mc{M}_G(C)$ for a smooth curve $C$.  Notice any $\ga \in G(C-p) = \hom_{Sch}(C - p, G)$ can be Laurent expanded around $p$ to produce an element in $LG$.  This realizes $G(C-p)$ as a subgroup of $LG$ which we denote $L_C G$.  Let $m_p \subset \mc{O}_{C,p}$ be the maximal ideal then choosing a basis $z \in m_p/m_p^2$ determines an isomorphism $\ec \hat{\mc{O}}_{C,p} \cong \ec \C[[z]] = D$. 

To make the connection between $LG$ and $\mc{M}_G(C)$ we introduce two functors.  Let $\C Alg$ denote the category of $\C$-algebras. Let $T' \colon \C Alg \to \mathbf{Set}$ be defined by setting $T'(R)$ to be the set of isomorphism classes of triples $(P,\tau_C,\tau_D)$ where $P$ is principal $G$-bundle on $C_R$, $\tau_D \colon G \x D_R \xrightarrow{\sim} P|_{D_R}$, $\tau_C \colon G \x (C-p)_R \xrightarrow{\sim} P|_{(C-p)_R}$ are trivializations.  Let $T$ be the functor defined by setting $T(R)$ to be isomorphism classes of pairs $(P,\tau_C)$ defined as above.  We have forgetful functors $T'\xrightarrow{f_D} T \xrightarrow{f_C} \mc{M}_G(C)$ defined by $(P, \tau_C,\tau_D) \xrightarrow{f_D} (P, \tau_C) \xrightarrow{f_C} P$.

Let $\tau_D^*$ denote the restriction of $\tau_D$ to $D^*_R = \ec R((z))$ and define $\tau_C^*$ similarly. Then we get a map
\begin{equation}\label{LGiso}
\begin{aligned}
 T' &\xrightarrow{\Theta_{C,D} } LG\\
(P,\tau_C,\tau_D) &\mapsto (\tau^*_C)^{-1} \circ \tau^*_D.
\end{aligned}
\end{equation}
Of course we also have $\Theta_{C,D}^{-1} \colon T' \to LG$ given by  $(P,\tau_C,\tau_D) \mapsto (\tau^*_D)^{-1} \circ \tau^*_C$. For definiteness we work with $\Theta_{C,D}$ but this choice is inconsequential.  

Denote by $LG/L^+G$ the sheaf associated to the pre sheaf $R \mapsto LG(R)/L^+G(R)$ in the fppf topology.  Then, for example, if $R \to R'$ is faithfully flat, $\ga \in LG(R')$, $\ga_1,\ga_2$ denote the images of $\ga$ under the two maps $R' \rightrightarrows R'\ox_R R' =: R''$ and $\ga_1 \ga_2^{-1} \in L^+G(R'')$ , then letting $\ol{\ga}$ denote the class of $\ga$  in $LG(R')/L^+G(R')$ we have $\ol{\ga_1} = \ol{\ga_2} \in LG(R'')/L^+G(R'')$.  By definition this determines a point of $(LG/L^+G)(R)$ which we denote $\ga \rightrightarrows (\ga_1, \ga_2)$.

We define a map $T \xrightarrow{\Theta_C} LG/L^+G$ as follows.  If $(P,\tau_C) \in T(R)$ then there is faithfully flat base extension $R \to R'$ such that $P|_{D_{R'}}$ admits a trivialization $\tau_D$ and hence a point of $T'(R')$. Let $\gamma(\tau_D) = \Theta_{C,D}(P,\tau_C,\tau_D) \in LG(R')$.  With $\ga_i(\tau_D)$ as above we set $\Theta_C((P,\tau_C))$ = $\ga(\tau_D) \rightrightarrows (\ga_1(\tau_D), \ga_2(\tau_D))$.  If $\tau'_D$ is another trivialization $\ga(\tau'_D),\ga_i(\tau'_D)$ differ from the unprimed version by elements in $L^+G$ hence define the same element in the quotient.

Similarly, we define a map $\mc{M}_G(C) \xrightarrow{\Theta} L_CG \backslash LG/L^+G$.  Let $P \in \mc{M}_G(C)(R)$ then by \cite{Drin}, there is a faithfully flat (in fact \'etale) base change $R \to R'$ such that $P|_{(C - p)_{R'}}$ admits a trivialization $\tau_C$ and hence a point of $T(R')$.  Let $\ga(\tau_C) = \Theta_C((P,\tau_C)) \in (LG/L^+G)(R')$ and $\ga_i(\tau_C)$ denote the two images of $\ga(\tau_C)$ in $(LG/L^+G)(R'\ox_R R')$.  Let $\ol{\ga(\tau_C)}$ denote the class of $\ga$ in $L_CG(R')\backslash (LG/L^+G)(R')$ and define $\ol{\ga_i(\tau_C)}$ similarly.  One checks $\ol{\ga_1(\tau_C)} = \ol{\ga_2(\tau_C)}$ and we set $\Theta(P) =  \ga(\tau_C) \rightrightarrows (\ga_1(\tau_C), \ga_2(\tau_C))$; as in the definition of $\Theta_C$, the map $\Theta$ is independent of the choice $\tau_C$.

Let $\pi_D \colon LG \to LG/L^+G$ and $\pi_C \colon LG/L^+G \to L_CG \backslash LG/L^+G$ be the quotient maps.  Summarizing, we have a commutative diagram
\begin{equation}\label{LGdiagram}
\xymatrix{
T'\ar[d]^{f_D}\ar[r]^{\Theta_{D,C}} & LG \ar[d]^{\pi_D}\\
T\ar[d]^{f_C}\ar[r]^{\Theta_C\ \ \ } & LG/L^+G \ar[d]^{\pi_C}\\
\mc{M}_G(C)\ar[r]^{\Theta\ \ \ \ \  \ } & L_CG \backslash LG/L^+G.
}
\end{equation}
We stress that while $\Theta_{D,C},\Theta_C$ are easy to construct, in order to construct $\Theta$ we need to use the non trivial result \cite{Drin} of Drinfeld and Simpson.  The construction of the maps $\Theta_{D,C}, \Theta_{C},\Theta$ we refer to collectively as the double coset construction (DCC).  The connection between $\mc{M}_G(C)$ and $LG$ can then be stated as

\begin{thm}\label{thm:DCC} 
All the horizontal maps in the diagram \eqref{LGdiagram} are isomorphisms.
\end{thm}

\begin{proof}
See \cite[Prop.3.4]{Bea} for details. For each map one constructs a map in the other directions and checks it is the required inverse.  For the inverse to $\Theta_{C,D}$, we construct for every $\ga \in LG(R)$ a $G$-bundle on $C_R$ with trivializations on $(C-p)_R,D_R$.  If $R$ is Noetherian then $(C-p)_R \sqcup D_R$ form an fppf cover of $C_R$ and standard descent allows us to glue the trivial $G$-bundles on $(C-p)_R$ and $D_R$ over $D^*_R = R((z))$ using $\ga$ as a transition function.  In  \cite{MR1320381} Beauville and Lazlo show such gluing is possible for an arbitrary $\C$-algebra $R$.  Alternatively, for any $\C$-algebra $R$ we can write $R = \varinjlim R_i$ with $R_i$ Noetherian.  By fixing an embedding $G \subset SL_n(\C)$ we can realize any $\ga \in LG(R)$ as an $n\x n$ matrix with entries in $R$.  Each entry lies in some $R_i$ and it follows there is a single $R_i$ such that $\ga \in LG(R_i)$.  Then we can apply fppf descent to obtain a $G$-bundle with trivializations on $C_{R_i}$ and pull everything back to $C_R$.

For the inverse to $\Theta_C$, let $\ol{\ga} \in (LG/L^+G)(R)$. Then there is a faithfully flat base change $R \to R'$ such that we can present $\ol{\ga}$ as $\ga \rightrightarrows (\ga_1, \ga_2)$ with $\ga \in LG(R')$ as discussed below \eqref{LGiso}. Set $\la = \ga_1 \ga_2^{-1}$.  Let $(P',\tau_C,\tau_D) = \Theta_{C,D}^{-1}(\ga)$ and let $(P''_i,\tau_{i,C},\tau_{i,D}) = \Theta_{C,D}^{-1}(\ga_i)$ for $i = 1,2$ be the two different pull backs to $C_{R''}$ where $R'' = R'\ox_R R'$.  The group $L^+G(R'')$ acts on $(P''_i,\tau_{i,C},\tau_{i,D})$ by changing the trivialization $\tau_{i,D}$.  We have $\la := \ga^{-1}_2 \ga_1 \in L^+G(R'')$ and evidently $(P''_1,\tau_{1,C},\tau_{1,D}) = (P''_2,\tau_{2,C},\tau_{2,D}) \la $.  Applying the forgetful map $f_D$ we see $(P''_1,\tau_{1,C}) = (P''_2,\tau_{2,C})$ in $T(R'')$ and therefore this data descends to $(P,\tau_C) \in T(R)$.  The argument for $\Theta$ is similar and omitted.
\end{proof}

We now describe a few variants of the DCC.  The descent lemma \cite{MR1320381} of Beauville and Lazlo in general will not apply to these variants but we can still argue by filtering by Noetherian subrings as in the proof above.

Suppose $\mc{M}$ is a moduli space of sheaves of sets on a smooth curve $C$ with a marked point $p$ such that for all $P \in \mc{M}$ we have $P|_{C-p} \in \mc{M}_G(C-p)$.  Suppose further that all objects are isomorphic over $D$ and the set of automorphisms of $P|_D$ is a subgroup $H \subset LG$.  Let $T_\mc{M}$ denote the moduli of space of pairs $(P,\tau)$ where $\tau$ is a trivialization of $P|_{C-p}$. Then the DCC yields maps
\begin{equation}\label{name.Theta}
\begin{aligned}
 T_\mc{M} &\xrightarrow{\Theta^H_{C} } LG/H\\
\mc{M} &\xrightarrow{\Theta^H} L_{C}G \backslash LG/H.
\end{aligned}
\end{equation}
For example, we can take $\mc{M}$ be the moduli space of quasi parabolic bundles with a reduction to a parabolic $Q \subset G$ at $p \in C$.  Then $H = L_Q^+G = \{\ga \in L^+G | \ga(0) \in Q\}$.

Consider a nodal curve $C_0$ with single node $p$; we have $\ec \hat{\mc{O}}_{C_0,p} \cong \ec \C[[x,y]]/xy = D_0$ so $D_0^* = \ec \C((x)) \x \C((y))$ and $LG \x LG$ takes the roles of $LG$ and $G^\Delta = G( \C[[x,y]]/xy)$ takes the role of $L^+G$.  The DCC yields 
\begin{equation}\label{eq:nodeDCC}
\begin{aligned}
 T &\xrightarrow{\Psi_{C_0} } L_xG \x L_yG/G^{\Delta}\\
\mc{M}_G(C_0) &\xrightarrow{\Psi} L_{C_0}G \backslash L_xG \x L_yG/G^\Delta .
\end{aligned}
\end{equation}

We can generalize as before to a moduli stack $\mc{M}$ of sheaves of sets on the nodal curve $C_0$ such that for all $P \in \mc{M}$ we have $P|_{C_0 - p} \in \mc{M}_G(C_0-p)$ and all objects are isomorphic over $D_0$ and $Aut(P|_{D_0})$ is a subgroup of $L_xG \x L_y G$.  Defining $T_\mc{M}$ in an analogous manner we obtain
\begin{equation}\label{name.Psi}
\begin{aligned}
 T_\mc{M} &\xrightarrow{\Psi^H_{C_0} } L_xG \x L_yG/H\\
\mc{M}_G(C_0) &\xrightarrow{\Psi^H} L_{C_0}G \backslash L_xG \x L_yG/H.
\end{aligned}
\end{equation}

For example, we could take $\mc{M}$ to be the moduli of quasi parabolic $G$-bundles with a reduction of the structure group to a parabolic $Q \subset G$ at the node $p$.  Then $H = \{(\ga_1, \ga_2) \in L_Q^+G \x L_Q^+G | \ga_1(0) = \ga_2(0)\}$ where $L_Q^+G = \{\ga \in L^+G | \ga(0) \in Q\}$.

Another variant is to take a twisted curve $C_{[k]}$ with twisted point $p$ and a smooth coarse moduli space $C$.  We choose a $k$th root $w$ of $z$ so that $C_{[k]}\x_C D = [\ec \C[[w]]/\mu_k]$ where $\zeta \in \mu_k$ acts by $w \mapsto \zeta w$. Let $\mu_k \xrightarrow{\eta} G$ be a homomorphism; the proof of lemma \ref{l:fixk} shows we can take this to be the restriction of a co-character $\Cs \xrightarrow{\eta} G$.  Then $\zeta \in \mu_k$ acts on $L_wG = G((w))$ by $g(w) \xrightarrow{\eta} \eta(\zeta)^{-1} g(\zeta w) \eta(\zeta)$; this action is explained in the proof of proposition \ref{p:para2equiv.smooth}.  Let $(L_wG)^{\mu_k}$ denote the invariants.  Then for $g(z) \in LG = G((z))$ the assignment $g(z) \mapsto g_\eta(w) := \eta(w) g(w^k) \eta(w)^{-1}$ defines an isomorphism $LG \xrightarrow{ \eta(w) (\ \ ) \eta(w)^{-1}} (L_wG)^{\mu_k}$ and in this way allows us to consider $L_CG \subset LG$ as a subgroup of $(L_wG)^{\mu_k}$.  

Let $\mc{M}_{G,\eta}(C_{[k]})$ be the moduli stack of $G$-bundles on $C_{[k]}$ with equivariant structure at $p$ determined by $\eta$. Let $T_{G,\eta}$ be the moduli of paris $(P,\tau)$ with $P \in \mc{M}_{G,\eta}(C_{[k]})$ and $\tau$ a trivialization of $P$ over $C_{[k]} - p$. The DCC yields
\begin{equation}\label{name.Theta.k}
\begin{aligned}
 T_{G,\eta} &\xrightarrow{\Theta^\eta_{C} } (L_wG)^{\mu_k}/(L_w^+G)^{\mu_k}\\
\mc{M}_{G,\eta}(C_{[k]}) &\xrightarrow{\Theta^\eta} (L_{C}G)^{\mu_k} \backslash(L_wG)^{\mu_k}/(L_w^+G)^{\mu_k}.
\end{aligned}
\end{equation}

Finally, we can consider a fixed twisted nodal curve $C_{0,[k]}$ with twisted node $p$ and coarse moduli space $C_0$. We choose $k$th roots $u,v$ of $x,y$ so that $C_{0,[k]} \x_{C_0} D_0 = [\ec \frac{\C[[u,v]]}{uv}  / \mu_k]$ where $\zeta \in \mu_k$ acts by $(u,v) \mapsto (\zeta u, \zeta^{-1} v)$.  We define $\mc{M}_{G,\eta}(C_{0,[k]}), T_{G,\eta}$ similarly as above.  The DCC yields
\begin{equation}\label{name.Psi.k}
\begin{aligned}
 T_{G,\eta} &\xrightarrow{\Psi^\eta_{C_0} } (L_uG \x L_vG)^{\mu_k}/(G^\Delta)^{\mu_k}\\
\mc{M}_G(C_0) &\xrightarrow{\Psi^\eta} (L_{C_0}G)^{\mu_k} \backslash(L_uG \x L_vG)^{\mu_k}/(G^\Delta)^{\mu_k}.
\end{aligned}
\end{equation} 

\section{Bundles on twisted curves and twisted chains}

Here we investigate $G$-bundles on twisted nodal curves.  The motivation to consider these objects comes from the valuative criterion for completenss.  Specifically it comes from the following local calculation.

Let $C_S$ be as in \ref{s:notation.curve}  and $f \colon S \to S$ any morphism.  Let $C_{S,f}$ denote the base change and $C_{S^*,f} := C_{S,f}\x_S S^*$.  The valuative criterion requires that we provide, for any $G$-bundle $P$ on the smooth curve $C_{S^*,f}$, an object $F$ on $C_{S,f}$ such that $F|_{C_{S^*,f}}$ is $P$; we assume that $F$ is at least a sheaf of sets.

By abuse of notation let $f$ denote also the map on rings $\C[[s]] \xrightarrow{f} \C[[s]]$.  Assuming $f(s) \ne 0$ we can, after suitable change of coordinates, normalize $f$ so that $f(s) = s^k$, with $k\ne 0$. We can further restrict to $k \ge 1$, otherwise $f$ maps to the generic point and the base change is a family of smooth curves. When $k\ge 2$, let $C_{S,[k]}$ denote the twisted curve obtained from $C_S$ by removing a formal disc $D_S$ around the node and gluing in the quotient stack $[D_S^{\frac{1}{k}}/\mu_k]$; see \ref{s:notation.curve} for definitions.  Then there is a map $C_{S,[k]} \to C_{S,f}$ realizing the latter as the coarse moduli space of the former.  By abuse of notation let $p \in C_{S,[k]}$ also denote the twisted node, then $C_{S,[k]} \to C_{S,f}$ restricts to an isomorphism $C_{S,[k]} - [p] \cong C_{S,f} - p$.

\begin{prop}\label{p:first1/2}
Let $p$ be the node in $C_{S,f}$. Let $P$ be a $G$-bundle on $C_{S^*,f}$.  There is a $G$-bundle $P'$ on $C_{S,f} - p$ extending $P$. If $k = 1$ then $P'$ extends to a $G$-bundle on $C_{S,f}$. If $k>1$ then there is a $G$-bundle $P''$ on $C_{S,[k]}$ that restricts to $P'$ under the isomorphism $C_{S,[k]} - [p] \cong C_{S,f} - p$.
\end{prop}
\begin{proof}
Let $K$ be the generic point of $C_{S^*}$.  In \cite[Corollary 1.5]{deJong} it is shown that $H^1(K, G) = 1$.  Therefore we can extend $P$ over the generic point of $C_0$ by taking it to be trivial in a neighborhood of this point.  Thus we have extended $P$ on the complement of a codimension $2$ subset.  The surface $C_{S,f} - p$ is always smooth and by \cite[Thm 6.13]{Sansuc} the $G$-bundle extends to all of $C_{S,f} - p$.  When $k = 1$ the surface $C_{S,f}$ is smooth and applying again \cite{Sansuc} covers this case.  

We now assume $k\ge 2$. By the above, it suffices to study $F$ in a neighborhood of the node.  So we restrict to $D_S = \ec \C[[x,y]]$ and then $D_{S,f} = \ec \C[[x,y,s]]/(xy - s^k)$.  The basic observation is that $\C[[x,y,s]]/(xy - s^k)$ is the ring of $\mu_k$ invariants in $\C[[u,v]]$ with action given by $\zeta (u,v) = (\zeta u, \zeta^{-1} v)$ and we identify $u^k = x, v^k = y, uv = s$.   By \cite[Prop. 3.7]{Brav}, every $G$-bundle on $D_{S,f} - p = \ec \C[[x,t]][x^{-1}] \cup \ec \C[[y, t]][y^{-1}]$ is the restriction of a $G$-bundle on $[\ec \C[[u,v]]/\mu_k]$; the result is stated for $\ec \C[x,y,s]/(xy - s^k)$ but the same proof works in our case. Consequently there is a $G$-bundle on $[D^{\frac{1}{k}}_S/\mu_k]$ that extends $P$ over the node.
\end{proof} 

\subsection{Fixed curve}
We now enter into an analysis of $G$-bundles on twisted curves. Let $C_{[k]}$ denote a twisted curve with smooth coarse moduli space $C$ and a single twisted point $p$ with stabilizer group $\mu_k$.   We show $G$-bundles $P$ on $C_{[k]}$ can be identified with torsors $\mc{F}$ on $C$ and that the moduli of such $\mc{F}$ on $C$ is not isomorphic to $\mc{M}_G(C)$.  This represents an obstruction to completing $\cMG(C_S)$ by only parametrizing degenerations of $G$-bundles on $C_0$; one should include degenerations of $G$-bundles on $C_{0,[k]}$ or degenerations of torsors on $C_0$.

Let $\eta \in \hom (\Cs, T) \ox_\Z \Q$ be an exotic co-character so that the associated parahoric $\mc{P} = P(\eta)$ is exotic ( see \ref{s:notation.para}). Let $\mc{G}^{std}$ be the sheaf of groups defined by $C \supset U \mapsto \hom_{Sch}(U, G)$.  Let $\cG^\mc{P}$ be the sheaf of group constructed in \ref{s:background.torsors}; namely $\cG^\mc{P}|_{C - p} = \cG^{std}$ and $\cG^\mc{P}(\hat \cO_{C,p}) = \mc{P}$.  

Let $T_{\cG^\mc{P}}(C)$ be the moduli space of pairs $(\mc{F},\tau)$ consisting of a $\cG^\mc{P}$-torsor $\mc{F}$ on $C$ together with a trivialization $\tau$ over $C-p$.  Similarly, let $T_{G,\eta}(C_{[k]})$ be the moduli space of pairs $(P,\tau)$ consisting of a $G$ bundle $P$ on $C_{[k]}$ with equviariant structure determined by $\eta$ (see \eqref{eq:equivariant.structure} in \ref{s:notation.root} and the paragraph below it) and a trivialization $\tau$ on $C-p$.  Define $T_{\cG^\mc{P}}(D)$ and $T_{G, \eta}([D^{\frac{1}{k}}/\mu_k])$ similarly with $C-p$ replaced by $D - p$. 
\begin{prop}\label{p:para2equiv.smooth}
Suppose $k \eta \in \hom (\Cs, T)$. Let $\cG^\mc{P}$,$C$,$C_{[k]}$, $D = \ec \C[[z]]$,$[D^{\frac{1}{k}}/\mu_k]$ be as above.  Choose a $k$th root $w$ or $z$ so that $D^{\frac{1}{k}} = \ec \C[[w]]$.  Let $i_{[k]} \colon [D^{\frac{1}{k}}/\mu_k] \to C_{[k]}$ and $i \colon D \to C$ be the natural maps. Then we have isomorphisms
\[
\xymatrix{
T_{\cG^\mc{P}}(D) & \ar[l]_{i^*}T_{\cG^\mc{P}}(C)\ar[d]_{\Theta_C^\mc{P}}\ar[r]^{\Xi_C} & T_{G, \eta}(C_{[k]})\ar[d]^{\Theta_C^\eta} \ar[r]^{ i_{[k]}^*} & T_{G, \eta}([D^{\frac{1}{k}}/\mu_k])\\
& LG/\mc{P}\ar[r]^{\eta (\ \ )\eta^{-1} \ \  \ \  \ \ } & (L_wG)^{\mu_k}/(L^+_wG)^{\mu_k} & 
}
\]
where $\Xi_C$ is defined to be $(\Theta_C^\eta)^{-1} \circ \eta (\ \ )\eta^{-1} \circ \Theta_C^\mc{P}$ and $\Theta_C^\mc{P}$ is the map in \eqref{name.Theta}, $\Theta_C^\eta$ the map in \eqref{name.Theta.k}, and $\eta (\ \ )\eta^{-1}$ is $g(z)\mc{P} \mapsto \eta(w) g(w^k)\eta(w)^{-1} (L^+_wG)^{\mu_k}$.
\end{prop}

\begin{proof}
Using descent theory as in the proof of theorem \ref{thm:DCC} we construct inverses to $i^*,i^*_{[k]}$. Let $(P_R,\tau) \in T_{\cG^\mc{P}}(D_R)$; after a flat base change $R \to R'$, the pullback of $P_R$ become trivial and comparing with $\tau$ defines a loop $\in LG(R')$. By gluing with the trivial bundle over $C-p$, we obtain a bundle with a fixed trivialization over $C-p \x \ec R$.  Again by descent theory this is well defined and gives an inverse map $T_{\cG^\mc{P}}(D) \to T_{\cG^\mc{P}}(C)$ similarly we have an inverse map $T_{G}([D^{\frac{1}{k}}/\mu_k]) \to T_G(C_{[k]})$. 

To establish that $\Theta^\mc{P},\Theta^\eta$ are isomorphisms it suffices to show their restrictions $T_{\cG^\mc{P}}(D) \to LG/\mc{P}$ and $T_{G, \eta}([D^{\frac{1}{k}}/\mu_k]) \to (L_wG)^{\mu_k}/(L^+_wG)^{\mu_k}$ define isomorphisms.  In the first case this follows because a point $LG/\mc{P}$ defines descent data for an object in $T_{\cG^\mc{P}}(D)$.

To handle the equivariant case we need to compute the $\mu_k$ equivariant automorphisms of $\ec \C((w)) \x G$  over $\ec \C((w))$. In order for $\ga \in L_w G = G((w))$ to define an equviariant automorphism of $\ec \C((w)) \x G$ (and thus determine an element of $(L_wG)^{\mu_k}$) we need for $\zeta \in \mu_k$
\[
\xymatrix{(w,g)\ar[d]^\gamma\ar[r]^\zeta & (\zeta w, \eta(\zeta)g)\ar[d]^\gamma\\
(w,\ga(w) g) \ar[r]^{\zeta \ \ \ \ \ \ \ \ \ \  \ \ \ \ \ \ \ \ \ } & (\zeta w, \eta(\zeta)\ga(w)g) =(\zeta w, \ga(\zeta w)\eta(\zeta)g)
}
\]
Or $\ga(w) = \eta(\zeta)^{-1}\ga(\zeta w) \eta(\zeta)$.  Thus we are concerned with invariants for the action of $\mu_k$ given by $\ga(w) \xrightarrow{\zeta} \eta(\zeta)^{-1}\ga(\zeta w) \eta(\zeta)$.  

We can now argue as before to establish $T_{G,\eta}([D^{\frac{1}{k}}/\mu_k]) \to G((w))^{\mu_k}/G[[w]]^{\mu_k}$ is an isomorphism.  It remains to check  $ LG/\mc{P} \to (L_wG)^{\mu_k}/(L^+_wG)^{\mu_k}$ is an isomorphism.

Let $\ga \in G((z))$ and set $\ga_\eta(w):= \eta(w) \ga(w^k) \eta(w)^{-1}$; the following shows $\ga_\eta(w) \in (L_wG)^{\mu_k}$:
\begin{align*}
\ga_\eta(w) &\xrightarrow{\zeta} \eta(\zeta)^{-1} \eta(\zeta w) \ga(w^k) \eta(\zeta w)^{-1}  \eta(\zeta)\\
&= \ga_\eta(w)
\end{align*}
where we have used that $(\zeta w)^k = w^k$ and $\eta(\zeta w) = \eta(\zeta) \eta(w)$.  Similarly, one can check for any $g(w) \in G((w))^{\mu_k}$ that $g^\eta(w) = \eta(w) g(w) \eta(w)^{-1} \in G((z))$ by checking it is invariant under the action $g^\eta(w) \mapsto g^\eta(\zeta w)$; thus $LG \xrightarrow{\eta (\ ) \eta^{-1}} ((L_wG)^{\mu_k})$ is an isomorphism.

Now let $\ga \in P(\eta)$.  We show in this case $\ga_\eta \in G[[w]]$.  It is sufficient to do this at the level of Lie algebras again because the groups involved are connected. In particular, $Lie(P(\eta))$ has a basis consisting of elements of the form $X_\a z^i$ where $X_\a$ is the root space associated to $\a$.  We have $\eta(w) X_\a z^i \eta(w)^{-1} = X_\a w^{k \l \eta,\a \r}z^i$.  Now the value of $\l \eta, \a\r$ is a rational number between $-1$ and $1$.  We can check if this is in $\mfg[[w]]$ by checking if $k \l \eta,\a \r + ki \ge 0$.  But this is equivalent to $ \l \eta,\a \r + i \ge 0$. Finally, $X_\a z^i \in Lie(P(\eta))$ implies that $\lim_{t \to 0} t^{\l \eta, \a\r + i} X_\a z^i$ exists which guarantees that $ \l \eta,\a \r + i \ge 0$. Altogether, we see $LG \xrightarrow{\eta (\ ) \eta^{-1}} ((L_wG)^{\mu_k})$ descends to an isomorphism as in the statement of the proposition.
\end{proof}
Let $\mc{M}_{G^{\mc{P}}}(C)$ be the moduli stack of $\cG^\mc{P}$-torsors on $C$ and $\mc{M}_{G,\eta}(C_{[k]})$ be the moduli space of $G$ bundle on $C_{[k]}$ with equivariant structure determined by $\eta$.

\begin{cor}\label{c:para2equiv.smooth}
Suppose $k \eta \in \hom(\Cs, T)$. The isomorphism $\Xi_C\colon T_{\cG^\mc{P}}(C) \to T_{G, \eta}(C_{[k]})$ of proposition \ref{p:para2equiv.smooth} descends to an isomorphism $\Xi \colon \mc{M}_{G^{\mc{P}}}(C) \to \mc{M}_{G,\eta}(C_{[k]})$
\end{cor}

\begin{proof}
In light of the previous proposition, the argument is purely formal and follows as in the proof of theorem \ref{thm:DCC}; see also \cite[prop.3.4]{Bea}.  

Let $P$ be a $G$-bundle on $\Ck$. The restriction of $P$ to $C-p$ is a $G$-bundle. By \cite{Drin} it is trivial.  Consequently the forgetful map $T_{G,\eta}(C_{[k]}) \to \mc{M}_{G,\eta}(C_{[k]})$ is essentially surjective and equivariant for the action of $L_CG = G[C-p]$ which changes the trivialization.  It descends to give a map $L_C G\backslash T_{G,\eta}(C_{[k]}) \to \mc{M}_{G,\eta}(C_{[k]})$ and one can construct an inverse by associating to $P$ the set of trivializations over $C-p$. The same argument holds for a $G^{\mc{P}}$-torsor $\mc{F}$ on $C$.  We obtain isomorphisms $\mc{M}_{G^{\mc{P}}}(C) \xrightarrow{\sim} L_CG \backslash LG/\mc{P}$ and $\mc{M}_{G,\eta}(\Ck) \xrightarrow{\sim} L_CG \backslash (L_wG)^{\mu_k}/(L^+_wG)^{\mu_k}$.  Finally, the isomorphism $LG/\mc{P} \xrightarrow{\eta (\ \ ) \eta^{-1}} (L_wG)^{\mu_k}/(L^+_wG)^{\mu_k}$ gives an isomorphism $L_CG \backslash LG/\mc{P} \to L_CG \backslash (L_wG)^{\mu_k}/(L^+_wG)^{\mu_k}$ which establishes the result.
\end{proof}

\begin{rmk} 
In \cite{Balaji}, Balaji and Seshadri develop similar results in the context of Bruhat-Tits group schemes.  For corollary \ref{c:para2equiv.smooth} see specifically \cite[Thm 5.2.2]{Balaji}.
\end{rmk}

\begin{rmk}\label{rmk:exotic}
For $G = SL_n$ all the parahorics of $LG$ are conjugate by elements in $LGL_n$ to subgroups $L_Q^+G \subset L^+G$ where $Q \subset G$ is a parabolic.  Consequently the resulting moduli spaces can be identified with moduli spaces of vector bundles with in general nontrivial determinant.  However in general the parahorics will no longer be even abstractly isomorphic and thus neither will be the resulting moduli spaces.  For example, $SP_4$ has a parahoric whose Levi factor is $SL_2 \x SL_2$ which distinguishes it from the standard parahoric $SP_4[[z]]$.  
\end{rmk}

\begin{rmk}\label{rmk:k_G}
Let $\eta_i$ be the vertices of $Al$.  Define $k_i$ as the minimum integer such that $k_i \cdot \eta_i \in \hom(\Cs, T)$ and set $k_G = lcm(k_i)$.  The $\eta_i$ correspond to the maximal parahorics $\mc{P}_i$ of $LG$ and further any parahoric $\mc{P}$ is conjugate to a subgroup of some $\mc{P}_i$. It follows readily that $k = k_G$ is the minimum value of $k$ for which the statement of corollary \ref{c:para2equiv.smooth} holds for any particular parahoric $\mc{P}$.
\end{rmk}

In section \ref{s:MainThm} we will need to fix the value of $k$; this is possible by remark \ref{rmk:k_G} and lemma \ref{l:fixk}. To state it we introduce some notation.  Let $i$ be a positive integer and set $C^*_i = \ec \C[[x,y,s]]/(x y - s^i) - (0,0,0)$.  For any two positive integers $i,j$ we can obtain $C^*_{ij}$ as a base change either from $C^*_i$ or $C^*_{j}$, that is the left commutative diagram induces the right commutative diagram:

\SelectTips{eu}{12}
\[
\xymatrix{
 \C[[s]] & \ar[l]^{s^{j} \mapsfrom s}\C[[s]]  & & C^*_{ij} \ar[r]\ar[d]  & C^*_i \ar[d] \\
 \C[[s]] \ar[u]^{s^{i} \mapsfrom s} & \ar[u]^{s^{i} \mapsfrom s}\ar[l]^{s^{j} \mapsfrom s}\C[[s]]  & &  C^*_{j} \ar[r] & C^*_1
} 
\] 

\begin{lemma}\label{l:fixk}
Let $k = k_G$ be as in remark \ref{rmk:k_G} and let $l$ be any positive integer. Let $P$ be a $G$-bundle on $C^*_l$.  Then there is a $G$-bundle $P'$ on $C^*_k$ such that $P \cong P'$ on $C^*_{l k}$. 
\end{lemma}
\begin{proof}
The curve $C^*_l$ has an $l$-fold cover $C^*_1 \to C^*_l$. By \cite[Prop. 3.7]{Brav}, a $G$-bundle $P$ on $C^*_l$ is equivalent to a $\mu_l$ equivariant $G$-bundle on $C^*_1$, which in turn is determined by a homomorphism $\mu_l \to G$. 

Let $\zeta \in \mu_l$ be a generator and $\mu_l \to G$ a homomorphism; by abuse of notation let $\zeta$ also denote the image in $G$.  Then $\zeta \in G$ is a semisimple element any by \cite[Thm. 22.2]{Humph}, $\zeta$ lies in a Borel subgroup; by\cite[Cor. 19.3]{Humph} it follows that $\zeta$ lies in a maximal torus $T$ and thus we can take $\mu_l \to G$ to be the restriction of a co-character $\eta \in \Cs \to T \subset G$, but for any such $\eta$, the co-character $\eta^l$ will always define the trivial action.  Thus, setting $\mft_\Z = \hom(\Cs, T)$ and $\mft_\Q = \hom(\Cs,T)\ox_\Z \Q$, we can take $\eta \in \mft_\Z/ l \cdot \mft_\Z \subset \mft_\Q/\mft_\Z$ where the inclusion is given by $\eta \mapsto \frac{1}{l}\eta$.

Further, identifying $\eta$ with $\Cs \xrightarrow{(id, \eta)} \Cs \x T \subset \Gsd$, we can also transform by the affine Weyl group $W^{aff}:= N_{\Gsd}(\Cs \x T)/(\Cs \x T)$ and thus assume $\frac{1}{l}\eta \in Al$.   For some subset $I\subset \{0, 1, \dotsc,r = rk(G)\}$,  we can express $\eta = \sum_{i \in I} a_i \eta_i$ with $a_i \in (0,1) \cap \Q$  and $\eta_i \in \mft_\Q$ the vertices of $Al$. 

Consider $\eta_I = \sum_{i \in I} \eta_i$; because $k \eta_i \in \hom(\Cs,T) \forall i$ we have that $\eta_I$ determines a $G$-bundle $P'$ on $C^*_k$.  We claim $P,P'$ pull back to isomorphic bundles on $C^*_{k l}$.  For this it suffices to show $T_{G,\frac{1}{l}\eta}([ D_S^{\frac{1}{k l}}/\mu_{k l}  ] ) \cong T_{G,\eta_I}( [ D_S^{\frac{1}{k l}}/\mu_{k l}  ] )$ and this in turn reduces to showing that the framed automorphism groups of $P,P'$ coincide. The automorphism groups are connected so it reduces to a Lie algebra calculation.  As these are subgroups of $G[[u,v]]$ (with $x = u^{k l}, y = v^{k l}$) the Lie algebra is spanned by formal sums of $X_\a u^i v^j$. If $i \ge j$ then this is $X_\a u^{i - j} (u v)^j$ and $u v$ is fixed by $\mu_{k l}\ $Êso we are reduced to the one variable case; we can argue analogously if $j \ge i$.  Then the claim about automorphism groups follows because $\eta_I$ and $\frac{1}{l}\eta$ lie in the interior of the same face of $Al$; follow the argument in the paragraph after \eqref{para.etaI} in section \ref{s:notation.Lie}.  
\end{proof}

\subsubsection{Fixed Nodal curve}
Let $C_{0,[k]}$ be a twisted nodal curve with a single twisted node $p$.  Let $C_0$ be it's coarse moduli space and by abuse of notation we also write $p \in C_0$ for the node.  The stabilizer of $p \in C_{0,[k]}$ is $\mu_k$ and in particular $C_{0,[k]} \x_{C_0} D_0 \cong [D^{\frac{1}{k}}_0/\mu_k] = $.

For an parahoric $\mc{P}$ let $\mc{L} \mc{U}$ be its Levi decomposition and set $\mc{P}^\Delta = \Delta(\mc{L})\ltimes (\mc{U} \x \mc{U})$.  Similarly as in \ref{s:background.torsors} one can construct a sheaf of groups $\mc{G}^\Delta$ over $C_0$ such that $\mc{G}^\Delta(\hat \cO_{C_0,p}) = \mc{P}^\Delta$ and $\mc{G}^\Delta|_{C_0 - p} = \mc{G}^{std}$.  Let $\mc{M}_{\mc{G}^\Delta}(C_0)$ denote the moduli stack of $\mc{G}^\Delta$ torsors on $C_0$ and let $T_{\mc{G}^\Delta}(C_0)$ denote the moduli space of pairs $(\mc{F},\tau)$ where $\mc{F} \in \mc{M}_{\mc{G}^\Delta}(C_0)$ and $\tau$ a trivialization of $\mc{F}$ over $C_0 - p$.  Define $T_{\mc{G}^\Delta}(C_0)$ similarly.

Let $\eta \in \hom(\Cs, T) \ox_\Z \Q$ and consider the moduli stack $\mc{M}_{G,\eta}(C_{0,[k]})$ of $G$-bundles on $C_{0,[k]}$ with equivariant structure at $p$ determined by $\eta$.  Let $T_{G,\eta}(C_{0,[k]})$ denote the moduli space of pairs $(P,\tau)$ with $P\in \mc{M}_{G,\eta}(C_{0,[k]})$ and $\tau$ a trivialization of $P$ on $C_{0,[k]} - p$.  Define $T_{G,\eta}( [D^{\frac{1}{k}}_0/\mu_k])$ similarly.

The arguments of proposition \ref{p:para2equiv.smooth} and corollary \ref{c:para2equiv.smooth} readily extend to nodal curves and we obtain
\begin{prop}\label{p:para2equiv.node}
Suppose $k \eta \in \hom (\Cs, T)$ and set $\mc{P} = \mc{P}(\eta)$. Let $\cG^\Delta$,$C_0$,$C_{0,[k]}$, $D_0 = \ec \C[[x,y]]/xy$, $[D_0^{\frac{1}{k}}/\mu_k]$ be as above.  Choose $k$th roots $u,v$ of $x,y$ so that $D^{\frac{1}{k}}_0 = \C[[u,v]]/uv$.  Let $i_{0,[k]} \colon [D_0^{\frac{1}{k}}/\mu_k] \to C_{0, [k]}$ and $i_0 \colon D_0 \to C_0$ be the natural maps.  Let $G^\Delta_{u,v} = \{(g_1,g_2) \in L^+_uG \x L^+_v G| g_1(0) = g_2(0)\}$. Then we have isomorphisms
\[
\xymatrix{
T_{\cG^\Delta}(D_0) & \ar[l]_{i_0^*}T_{\cG^\Delta}(C_0)\ar[d]_{ \Psi_C^{\mc{P}^\Delta}}\ar[r]^{\Xi_{C_0}} & T_{G, \eta}(C_{0,[k]})\ar[d]^{ \Psi_C^\eta} \ar[r]^{i_{0,[k]}^*} & T_{G, \eta}([D_0^{\frac{1}{k}}/\mu_k])\\
& LG \x LG/\mc{P}^\Delta \ar[r]^{\eta (\ \ )\eta^{-1} \ \  \ \  \ \ } & (L_uG \x L_vG)^{\mu_k}/(G^\Delta_{u,v})^{\mu_k} & 
}
\]
where $\Xi_{C_0}$ is defined to be $(\Psi_C^\eta)^{-1}\circ \eta (\ \ )\eta^{-1} \circ \Psi_C^{\mc{P}^\Delta}$ and $\Psi_C^{\mc{P}^\Delta}$ is the map in \eqref{name.Psi}, $\Psi_C^\eta$ the map in \eqref{name.Psi.k}, and the last map is the product of $g(z)\mc{P} \mapsto \eta(w) g(w^k)\eta^{-1}(w) (L^+_wG)^{\mu_k}$.  The isomorphism $\Xi_C$ descends to an isomorphism $\Xi \colon \mc{M}_{\cG^\mc{P}}(C_0) \to \mc{M}_{G, \eta}(C_{0,[k]})$. 

\hfill $\square$
\end{prop}

\subsubsection{Connection with $\ol{\pGsd}$}
In \cite{Solis} a stacky embedding $\ol{\pGsd}$ of ${\pGsd}$ was constructed which is analogous to the wonderful compactification of a semisimple adjoint group.  In particular, the boundary $\ol{\pGsd} - \pGsd$ is a divisor with simple normal crossings.  In \cite{Solis} it was also shown that some of the orbits in the boundary are naturally isomorphic to moduli spaces of torsors on a nodal curve.  To be more precise, let $r = rk(G)$. It is shown that the boundary $\ol{\Gsd} - \Gsd$ breaks up into a union of $2^{r+1}$ orbits $\mathbf{O}_I$ labeled by the subsets $I$ of $\{0, \dotsc, r+1\}$ and $\mathbf{O}_I$ is further described by:

\begin{prop}\label{p:orb_I}
Let $L_I, \mc{P}^\pm_I, \mc{U}^\pm_I$ be as in \eqref{para.etaI} in section \ref{s:notation.Lie} and let $Z_0(L_I)$ be the connected component of the center $Z(L_I)$ of $L_I$.  Define $\mc{P}_I^{\Delta, \pm} = \Delta(L_I) \ltimes (\mc{U}_I \x \mc{U}_I^-)$. We have
\[
\mathbf{O}_I =
\frac{\pLG \x \pLG}{Z_0(L_I) \x Z_0(L_I) \cdot \mc{P}_I^{\Delta,\pm}}
\]
In particular, the orbit $\mathbf{O}_I$ fibers over $LG/\mc{P}_I \x LG/\mc{P}^-_I$ with fiber $L_{I,ad} = L_I/Z_0(L_I)$.  Further, when $I$ is a singleton set the group $Z_0(L_I)$ is trivial and when $I$ has cardinality $>1$ we have $Z_0(L_I) = Z(L_I)$.
\end{prop}
The isomorphisms of proposition \ref{p:para2equiv.node} allows us to identify $\mathbf{O}_{\{i\}}$ with $T_{G,\eta_i}([D_0^{\frac{1}{k}}/\mu_k])$, $T_{G,\eta_i}(C_{0,[k]})$; here $\eta_i$ is the $i$th vertex of $Al$. The natural expectation is that $T_{G,\eta_i}([D_0^{\frac{1}{k}}/\mu_k])$ can further degenerate to a moduli problem parametrized by the higher co-dimensional orbits in $\ol{\Cs \ltimes \pLG}$ and similarly with $T_{G,\eta_i}(C_{0,[k]})$.  We show that this is indeed the case in the next sub section.

\subsection{$G$-bundles on Twisted Chains}
In the previous section we saw that associated to the singleton sets $\{i\} \subset \{0, r+1\}$ there is a moduli space parametrizing $G$-bundles on a twisted nodal curve and further the moduli space can be identified with an orbit of the wonderful embedding of the loop group.  In this section we introduce a more general moduli problem which we show is isomorphic to the orbit $\mathbf{O}_I$ in the wonderful embedding for any $I \subset \{0, \dotsc, r+1\}$.

Let $R_n$ denote the rational chain of projective lines with $n$-components; figure \ref{chain3} in the introduction depicts a chain of length $3$.  There is an action of $\Cs$ on $R_n$ which scales each component.  Let $p_0, \dotsc, p_n$ denote the fixed points of this action.

Recall $u,v$ are $k$th roots of $x,y$ which are our coordinates near a node.  Let $p',p''$ be denote the closed points of $\ec \C[[u]],\ \ec \C[[v]]$ and finally let $D^{\frac{1}{k}}_n$ be the curve obtained from $\ec \C[[u]] \sqcup R_n \sqcup \ec \C[[v]]$ by identifying $p'$ with $p_0$ and $p''$ with $p_n$.

The group $\mu_k$ acts on $D^{\frac{1}{k}}_n$ through its usual action on $u,v$ and through the inclusion $\mu_k \subset \Cs$ on the chain $R_n$.  For an $n$-tuple $(\beta_0, \dotsc, \beta_n) \in \hom(\Cs,T)^n$, we can speak about the equivariant $G$-bundles on $D^{\frac{1}{k}}_n$ with equivariant structure at $p_i$ determined by $\beta_i$.  We refer to this equivalently as a $G$-bundles on $[D^{\frac{1}{k}}_n/\mu_k]$ of type $(\beta_1, \dotsc, \beta_n)$.

Further, we can also glue $[D^{\frac{1}{k}}_n/\mu_k]$ to $C_0 - p_0$ to obtain a curve $C_{n,[k]}$.  Let $C_n$ denote the coarse moduli space of $C_{n,[k]}$.
\begin{equation}\label{Cnk}
\text{We call $C_n$ a {\it modification} of $C_0$ and $C_{n,[k]}$ a {\it twisted modification} of $C_0$.}
\end{equation}

Recall the specific co-characters $\eta_0, \dotsc, \eta_r$ defined in \eqref{para.eta} in \ref{s:notation.para}. For $I = \{i_1, \dotsc, i_n\} \subset \{0, \dotsc, r\}$, let $T_{G,I}([D^{\frac{1}{k}}_n/\mu_k])$ denote the moduli space of pairs $(P, \tau)$ where $P$ is a $G$-bundles on $[D^{\frac{1}{k}}_n/\mu_k]$ of type $(\eta_{i_1}, \dotsc, \eta_{i_n}) $ and $\tau$ is a trivialization on $[\ec \C((u)) \x \C((v))/\mu_k]$. Let $H = Aut(P)$ then restriction to $\ec \C[[u]]$ and $\ec \C[[v]]$ realizes $H \subset (L_uG)^{\mu_k} \x (L_uG)^{\mu_k}$.

\begin{thm}\label{thm:Gbundles.chains.para.represented}
Let $I  \subset \{0, \dotsc, r\}$ and $T_{G,I}([D^{\frac{1}{k}}_n/\mu_k])$ be as above.  Then there is an isomorphism

\[
T_{G,I}(C_{0,[k]}) \xrightarrow{\Psi^{\eta_I}} (L_uG)^{\mu_k} \x (L_uG)^{\mu_k}/H \xrightarrow{\eta_I^{-1}(\ \ ) \eta_I}
\frac{\pLG \x \pLG}{Z(L_I) \x Z(L_I) \cdot \mc{P}_I^{\Delta,\pm}}.
\]
where $\Psi^{\eta_I}$ is as in \eqref{name.Psi.k} and $\eta_I^{-1}(\ \ ) \eta_I$ is described in proposition \ref{p:para2equiv.node}.  Let $i \colon [D^{\frac{1}{k}}_n/\mu_k] \to C_{0,[k]}$ be the natural map. Then $i^* \colon T_{G,I}(C_{0,[k]}) \to [D^{\frac{1}{k}}_n/\mu_k]$ is an isomorphism. In particular, $T_{G,I}(C_{0,[k]})$, $T_{G,I}([D^{\frac{1}{k}}_n/\mu_k])$ are isomorphic to an orbit in the wonderful embedding of $\ol{\pGsd}$.
\end{thm}

\begin{proof}
That $i^*$ is an isomorphism follows formally so we focus on showing that $T_{G,I}([D^{\frac{1}{k}}_n/\mu_k])$ is isomorphic to the stated homogeneous space. We suppress the isomorphism $\eta_I^{-1}(\ \ ) \eta_I$ and work inside $G((x)) \x G((y))$ with the help of the identification $[\ec \C((u)) \x \C((v))/\mu_k] = \ec \C((x)) \x \C((y))$.

The strategy is the same as in the proof of proposition \ref{p:para2equiv.smooth} and corollary \ref{c:para2equiv.smooth}  above. Namely, fix an object $(P,\tau)$ of $T_{G,I}([D^{\frac{1}{k}}_n/\mu_k]) $. The restriction of $P$ to $\ec \C[[x]] \sqcup \ec \C[[y]]$ is necessarily trivial and comparing with $\tau$ produces loops in $G((x)) \x G((y)) = LG \x LG$.  Loops are identified that differ by an automorphism of $P$ over $[D^{\frac{1}{k}}_n/\mu_k]$; that is, an element of $H$.  We will show $H \cong Z(L_I) \x Z(L_I)\cdot  \mc{P}_I^\Delta \subset LG \x LG$.  Then we notice that
\[
\frac{LG \x LG}{Z(L_I) \x Z(L_I) \cdot \mc{P}_I^\Delta} \cong \frac{\pLG \x \pLG}{Z(L_I) \x Z(L_I)\cdot \mc{P}_I^{\Delta,\pm}},
\]
The above isomorphism holds because for $\mc{P}_{I,poly} = \mc{P}_I \cap \pLG$ we have $LG/\mc{P}_I = \pLG/\mc{P}_{I,poly} \cong \pLG/\mc{P}^-_{I,poly}$; these statements are proved in \cite[7.4]{Kumar}. 

We turn now to computing $H = Aut(P)$. Let $H_u = Aut(P|_{[\ec \C[[u]]/\mu_k]})$, $H_n = Aut(P|_{[R_n/\mu_k]})$ and $H_v = Aut(P|_{[\ec \C[[v]]/\mu_k]})$.  Let $ev_u \colon H_u \to G$ be the restriction of an automorphism to the special point; define $ev_v$ similarly.  Finally let $ev_{0,n} \colon H_n \to G \x G$ be the restriction of an automorphism to the two extreme points of $[R_n/\mu_k]$.  Then we have $H = \{(f_u,f_n,f_v) | ( ev_u(f_u), ev_v(f_v) ) = ev_n (f_n) \} \subset H_u \x H_n \x H_v$.

By \ref{p:para2equiv.smooth}, we have $H_u \x H_v = P(\eta_{i_0}) \x P(\eta_{i_n})$.

We now compute $H_n$. Let $E = E(\eta_{i_1}, \dotsc, \eta_{i_n})$ denote $P|_{[R_n/\mu_k]}$.  In fact, automorphisms of $G$-bundles on $[R_n/\Cs]$ have been computed by Martens and Thaddeus in \cite{Martens}. They consider a slightly different situation where they fix $\eta_{i_0} = \eta_{i_n} = 0$, but we can still use the same methods to handle our case.  

Then results of \cite{MartensThaddeus} imply $H_n$ is connected so we pass to $Lie(H_n) = H^0([R_n/\mu_k], ad E)$.  Let $ev_{0,n}$ denote also the map on Lie algebras $ev_{0,n} \colon H^0([R_n/\mu_k], ad E) \to \mfg \oplus \mfg$. The map $ev_{0,n}$ is embedding because $\ker ev_{0,n} = H^0([R_n/\mu_k], ad E(\eta_I)\ox \cO(-p_0 -p_n)) = 0$ by lemma \ref{l:H0=0}.  For a tuple of integers $(b_0, \dotsc, b_n)$ let $\cO(b_0, \dotsc, b_n)$ denote line bundle on $[R_n/\mu_k]$ with equivariant structure at the fixed points $p_i$ given by $b_i$. Then $ad E \cong \oplus_{i = 1}^{rk(G)} \cO(0, \dotsc, 0) \oplus_{\a \in \Delta} \cO(\a \circ \eta_{i_1} , \dotsc, \a \circ \eta_{i_n})$ and we can compute separately for each $\a$.

For $\a$ a root of $L_I$ we have $\a \cdot \eta_{i_j} = 0$ and these roots contribute a factor of $\Delta(Lie(L_I))$ to the image of $ev_{0,n}$.  If $\a$ is negative, $\l \a, \eta_{i_0}\r = 0$, and some other $\l \a, \eta_{i_j} \r < 0$ then there is a consecutive subset $\{ i_0, i_1,\dotsc, i_{j'}\}$ such that $0 = \l \eta_{i_0}, \a\r = \dotsc = \l \eta_{i_{j' -1}} \r$ and $\l \eta_{i_{j'}}, \a \r <0$.  Then  \cite[1.2(c)]{Martens} implies $(X_\a, 0) \in \mfg \oplus \mfg$ lies in the image of $ev_{0,n}$; in fact \cite[1.2(c)]{Martens} ensures there is such a $\Cs$-invariant section which is then necessarily $\mu_k$ invariant.  Similarly, if $\a$ is positive, $\l \a , \eta_{i_n} \r = 0$ and some other $\l  \a, \eta_{i_j} \r > 0$ then the image contains $(0, X_\a)$.  

There is a second contribution to the group $H_n$. Namely, we can lift $Aut(R_n) = (\Cs)^n$ to $E$.  Describe $R_n = \cup_{i=1}^n C_i$ as a chain of $\P^1$s going from left to right with fixed points $p_{j-1}, p_j \in C_j$ on the $j$th component. Let $(\Cs)_j$ be the $j$th $\Cs$ factor in $Aut(R_n)$. Then lifting $(\Cs)_j$ to $E$ requires a homomorphism $\Cs \to G$ for each fixed point.  This homomorphism must be $\eta_{i_{j-1}}$ at $p_{j-1}$ and by continuity it must also be $\eta_{i_{j-1}}$ on all $C_i$ with $i < j$.  Similarly the lifting is determined by $\eta_{i_j}$ on all $C_i$ with $i>j$.

Let $P_I \subset G$ be the parabolic associated to the co-character $\eta_I = \sum_{i_j \in I} \eta_{i_j}$ and let $L_I U_I$ be its Levi decomposition. 
Each $\eta_{i_j}$ maps into $Z(L_I)$ and under $ev_{0,n}$ generates a complement to $\Delta(Z(L_I)) \subset Z(L_I) \x Z(L_I)$.  Also $U_I$ consists of those $X_\a$ with $\a >0$ such that $\l \eta_{i_j}, \a\r > 0$ for some $i_j$. Altogether, we get that $H_n = Z(L_I) \x Z(L_I) \cdot \Delta(L_I) \ltimes (U^-_I \x U_I)$.

Consulting \eqref{para.etaI} in section \ref{s:notation.para} and comparing the computations of $H_u,H_v,H_n$, we conclude that $H = Z(L_I) \x Z(L_I) \cdot L_I \ltimes (\mc{U}_I^- \x \mc{U}_I) \cong Z(L_I) \x Z(L_I) \cdot \mc{P}_I^\Delta$.
\end{proof}

\begin{lemma}\label{l:H0=0}
For $\{i_0, \dotsc, i_n\} \subset \{0,\dotsc, r\}$ let $E = E(\eta_{i_0}, \dotsc, \eta_{i_n})$ be the $G$-bundle with splitting type $(\eta_{i_0}, \dotsc, \eta_{i_n})$.  Then $H^0([R_n/\mu_k],ad E \ox \cO(-p_0 - p_n)) = 0$
\end{lemma}

\begin{proof}
For a tuple of integers $(b_0, \dotsc, b_n)$ let $\cO(b_0, \dotsc, b_n)$ denote line bundle on $[R_n/\mu_k]$ with equivariant structure at $p_i$ given by $b_i$. Then $\cO(-p_0 - p_n) = \cO(-1,0,\cdots, 0,1)$.  We remind the reader that a single subscript $\eta_l$ denotes a specific co-character with $l$ ranging from $\{0,\dotsc, r\}$ and double subscripts $\eta_{i_j}$ are used to denote ordered subsets $\{i_1, \dotsc, i_n\} \subset \{0,\dotsc, r\}$.

We have $ad E = \oplus_{i = 1}^{rk(G)} \cO(0,\dotsc, 0) \oplus_{\a \in \Delta} \cO(\a \circ \eta_{i_0}, \dotsc, \a \circ \eta_{i_n})$.   Clearly the trivial summand poses no problem.  By symmetry we can focus on $\a$ positive, in which case we show that
\[
H^0([R_n/\mu_k],\cO(\a \circ \eta_{i_0}, \dotsc, \a \circ \eta_{i_n}) \ox  \cO(-p_0 - p_n) ) = H^0([R_n/\mu_k], \cO(\a \circ \eta_{i_0}-1, \dotsc, \a \circ \eta_{i_n}+1) ) = 0.
\]
Because all the $\eta_i$ are in the Weyl alcove we have all $\a \circ \eta_{i_j}\ge 0$.  Also for the longest root $\theta = \sum_i n_i \a_i$ we have $1 = \theta \circ \eta_j = \sum_i n_i (\a_i \circ \eta_j)$ and all the $n_i\ge 1$.  This implies $\a_i \circ \eta_j = \frac{1}{n_i}\delta_{i,j}$.

Express our fixed $\a = \sum_i m_i \a_i$ with $0\le m_i\le n_i$. Let us first establish some properties of $\cO(\a \circ \eta_l,\a \circ \eta_{l'})$ on $[\P^1/\mu_k]$.  The degree of $\cO(\a \circ \eta_l,\a \circ \eta_{l'})$ is 
\begin{equation}\label{<= k}
d = k\a\circ (\eta_l - \eta_{l'}) = k(\frac{m_I}{n_l} - \frac{m_{l'}}{n_{l'}})
\end{equation}
and we note $d$ is an integer with $|d| \le k$.  Provided $d \ge 0$ the global sections are spanned by the $\mu_k$-invariant monomials $x_0^{d-d'}x^{d'}_1$ where $x_0$ has weight $\frac{k}{n_l}$ and $x_1$ has weight $\frac{k}{n_{l'}}$.  

We examine the restriction of $\cO(\a \circ \eta_{i_0}, \dotsc, \a \circ \eta_{i_n}) \ox  \cO(-p_0 - p_n)$ to various components. Clearly we can restrict to those components where the degree is $d>0$.  Below, when we restrict to the $j$ component $[\P^1/\mu_k]$ we set $l = {i_{j-1}}$ and $l' = i_j$ .

Suppose we restrict to a component with $d,m_{l'}>0$ then $n_l\ge m_l>d$ so all monomials of degree $d$ have nonzero weight provided the weight is less than $k$.  This holds because $x_0^d$ has the highest weight and it is $d\frac{k}{n_l}$ which is less than $k$. Consequently there are no sections.

We now assume $m_{l'} = 0$. If we restrict to the first component then tensoring with $\cO(-p_0 - p_n)$ lowers the degree by $1$ and again there are no sections.  Otherwise, if the degree of the $j$th component of the bundle is $k$ then there are two sections $x_0^k,x_1^k$ that we must show cannot extend to a global section.  Assume $x_1^k$ is non vanishing at $p_j$.  The degree on the $j+1$ component is either $0$ or negative. If the degree is $0$ then on the $j+2$ component the degree is either $0$ or negative.  Thanks to tensoring with $\cO(-p_0 - p_n)$ we are certain to eventually get a negative bundle which has no sections.  Therefore the section $x_1^k$ cannot extend.  But to extend the section $x_0^k$ on the $j-1$ component we need a section on bundle with degree $d>0$ and  $m_{l'}>0$ which is impossible by the previous paragraph.
\end{proof}

\begin{rmk}
For comparison with the $\Cs$ equivariant automorphisms see \cite[2.13,2.19]{Martens}.   
\end{rmk}

\begin{rmk}
In \eqref{<= k} we concluded that the degrees of the bundle on the chain have to be bounded by $k$.  It is worth noting that this recovers the moduli problem considered for $GL_r$ by Kausz \cite{K2}.  In this case one can work on non twisted curves; that is, with $k = 1$.  Then Gieseker bundles are exactly vector bundles on modifications of the curve such that the restriction to a chain splits as a direct sum $\cO$ and $\cO(1)$ and $H^0(R_n, E(-p'-p'')) = 0$. The latter condition implies $H^0(R_n,ad E(-p' - p'')) = 0$.
\end{rmk}

\section{Twisted Gieseker Bundles}\label{s:MainThm}
In this section we begin with a curve $C_S$ as in section \ref{s:notation.curve} and construct an algebraic $S$-stack $\mc{X}_G(C_S)$ such that $\mc{M}_G(C_S) \subset \mc{X}_G(C_S)$ is a dense open substack and the boundary is a divisor with normal crossings.  Further we show the morphism $\mc{X}_G(C_S) \to S$ is complete.

For the remainder of this section we fix a simple group $G$ as in section \ref{s:notation.Lie} and further fix an integer $k = k_G$ as in remark \ref{rmk:k_G}.  The only exception is proposition \ref{p:Mdf} where $k$ can be any integer $\ge 1$.

For convenience, we recall some of the notation from \ref{s:notation.curve}.  Namely, $S = \ec \C[[s]]$, $S^* = \ec \C((s))$, $S_0 = \ec \C[[s]]/(s) = \ec \C$, $C_0 = C_{S_0}$.  For $B$ an $S$-scheme we set $B^* = B\x_S S^*$, $B_0 = B \x_S S_0$.  We also have $D_S = \ec \C[[x,y]]$ considered as an $S$-scheme via $s \mapsto x y$ and $D_0 = \ec \C[[x,y]]/(xy)$. Further, we set $D_S^{\frac{1}{k}}:= \C[[u,v]]$ where $u^k = x$ and $v^k = y$.  Then $D_{S,[k]} = [D^{\frac{1}{k}}_S/\mu_k]$; the coarse moduli space of $D_{S,[k]}$ is $\ec \C[[x,y,s]]/(xy - s^k)$.  We further fix $p \in C_S$ to be the node.

To define $\mc{X}_G(C_S)$ we need to define twisted modifications of $C_S$; this is a relative version of \eqref{Cnk}.  Then in subsection \ref{s:completeness} we define $\mc{X}_G(C_S)$ to be the moduli stack parametrizing $G$-bundles on twisted modifications.  There we prove the main theorem which shows that $\mc{X}_G(C_S)$ satisfies the valuative criterion for completeness.

\subsection{Twisted Modifications}\label{s:orb.modification}
Let $C_S$ be a nodal curve.  A {\it modification of length $\le n$} of $C_S$ over $B$ is a curve $C'_B$ over $B$ with a morphism $C'_B \xrightarrow{\pi} C_B$ such that
\begin{itemize}
\item $C'_B$ is flat over $B$ and $\pi$ is finitely presented and projective
\item $C'_{B^*} \xrightarrow{\pi} C_{B^*}$ is an isomorphism
\item for $b \in B_0$ the map of curves $C'_b \xrightarrow{\pi} C_b$ is a modification; that is the fiber $\pi^{-1}(p_b)$ over the unique node $p_b \in C_b$ is a rational chain of $\P^1$s with at most $n$ components and there is $b\in B_0$ such that $\pi^{-1}(p_b)$ has exactly $n$ components.
\end{itemize}

\begin{ex}
In \cite[4.2]{Gieseker}, Gieseker constructs modifications of length $n$ over $B= \C[[t_1, \dotsc, t_{n+1}]]$ mapping to $\C[[s]]$ via $s \mapsto t_1 \dotsb t_{n+1}$ and where the maximum number of components in the modification is reached only over $(0, \dotsc, 0)$.  Further, the $i$th node is locally described by $B[[x,y]]/(x y - t_i)$.  

We recall the construction for $n = 1$; it is sufficient for our purposes to work with the curve $D_S = \C[[x,y]]$.  The base change to $\C[[t_1,t_2]]$ is $D_{[[t_1,t_2]]} := \ec \C[[t_1,t_2,x,y]]/(x y - t_1t_2)$ and the modification $D'_{[[t_1,t_2]]}$ is the blow up of $D_{[[t_1,t_2]]}$ along the ideal $(x,t_1)$. The fibers of the map $D'_{[[t_1,t_2]]} \to \ec \C[[t_1,t_2]]$ agree with those of $D_{[[t_1,t_2]]}$ except over $(0,0)$ where the node has been replaced by a chain of length $1$.

By a series of analogous blowups we obtain a modification $D'_{[[t_1, \dotsc, t_{n+1}]]} \xrightarrow{f} D_{[[t_1, \dotsc, t_{n+1}]]}$ of $D_S$ over $\ec \C[[t_1, \dotsc, t_{n+1}]]$ such that for $\emptyset \ne I \subset \{ 1, \dotsc, n+1\}$ the fiber of $f$ over $\{ t_i = 0 \}_{i \in I}$ is a modification $D_{| I | - 1}$ of $D_0$ of length $|I| -1$.  This local construction extends to give a modification $C'_{[[t_1, \dotsc, t_{n+1}]]}$ of $C_S$ over $\ec \C[[t_1, \dotsc, t_{n+1}]]$. Gieseker in fact proves this construction gives a versal deformation of the curve $C_n$ in \eqref{Cnk}.  We utilize this in the proof of theorem \ref{thm:isAlgStack}.
\end{ex}

Let $(g_1, \dotsc, g_n) \in (\Cs)^n$ act on $ \C[[t_1, \dotsc, t_{n+1}]]$ by $(t_1,\dotsc, t_n) \xrightarrow{(g_1,\dotsc,g_n)} (g_1 t_1, \frac{g_2}{g_{1}} t_2, \dotsc, \frac{g_{n}}{g_{n-1}} t_n, \frac{1}{g_n} t_{n-1})$.  This action extends to $C'_{[[t_1, \dotsc, t_{n+1}]]}$ such that for every closed point $q \in \ec \C[[t_1, \dotsc, t_{n+1}]]$ the stabilizer of $q$ in $(\Cs)^n$ coincides with $Aut(C'_q/C_q)$.  We set $Mdf_n = [ \C[[t_1, \dotsc, t_{n+1}]]/(\Cs)^n]$.  This is an algebraic $S$-stack that comes equipped with a curve $[C'_{[[t_1, \dotsc, t_{n+1}]]}/(\Cs)^n]$ and the modifications of $C_S$ over $B$ that arise from $S$-maps $B \to Mdf_n$ we call {\it local modifications} of length $\le n$.

A  {\it twisted modification} of length $\le n$ of $C_S$ over $B$ is a twisted curve $\mc{C}'_B$ such that its coarse moduli space $C'_B$ is a modification of length $\le n$ of $C_S$ over $B$. A twisted modification is of {\it order $k$} if the order of the stabilizer group of every twisted point has order exactly $k$.  Similarly, a twisted modification is of {\it order $\le k$} if the order of the stabilizer of every twisted point has order $\le k$.  A {\it local twisted modification} $\mc{C}'_B$ is a twisted modification whose coarse moduli space $C'_B$ is a local modification.  In the rest of this paper we work primarily with (twisted) local modifications.

\begin{rmk}
Restricting to local modifications is probably unnecessary but it simplifies our arguments and is sufficient to prove the main theorem.
\end{rmk}

Let $Mdf^{tw}_n$ denote the functor that assigns to $B \to S$ the groupoid of twisted local modifications of $C_S$ over $B$ of length $\le n$. Let $Mdf^{tw,k}_n \subset Mdf^{tw,\le k}_n$ be the functors of twisted local modifications of order $k$ and order $\le k$ respectively.  
\begin{prop}\label{p:Mdf}
Let $k\ge 1, n\ge 0$ be integers.  The functors $Mdf^{tw}_n$ and $Mdf^{tw,\le k}_n$ are algebraic stacks.  Further $Mdf^{tw,\le k}_n \subset Mdf^{tw}_n$ is an open substack and $Mdf^{tw,k}_n\subset Mdf^{tw}_n$ is a closed algebraic substack. Further, all of these stacks are locally of finite type.
\end{prop}

\begin{proof}
The basic tool is to use the stack of all genus $g$ curves.  For an integer $g$ let $\mc{S}_g$ denote the functor on $Sch$ which to any scheme $B$ assigns the groupoid of all (not necessarily stable) genus $g$ nodal curves $C \to B$.  In \cite[A]{MR2786662} it is shown that $\mc{S}_g$ is an algebraic stack locally of finite type; see also \cite[\textsection 5]{MR2309994}.  If $C'_B \to C_B$ is a local modification then forgetting the map to $C_B$ defines a morphism $Mdf \to \mc{S}_g$.

Let $\mathfrak{M}^{tw}_g$ be the functor which to any scheme $B$ assigns a genus $g$ twisted curve $\mc{C} \to B$.  In \cite[A]{MR2786662} it is shown that $\mathfrak{M}^{tw}_g$ is algebraic with a representable map to $\mc{S}_g$.  Further the sub functor $\mathfrak{M}^{tw,\le k}_g \subset \mathfrak{M}^{tw}_g$ of twisted curves with twisting or order $\le k$ is an open algebraic substack.  Then the result follows from
\begin{align*}
Mdf^{tw}_n &= Mdf_n \x_{\mc{S}_g} \mathfrak{M}^{tw}_g\\
Mdf^{tw,\le k}_n &= Mdf_n \x_{\mc{S}_g} \mathfrak{M}^{tw}_g,
\end{align*}
and that $Mdf^{tw, k}_n$ is the closed substack $Mdf^{tw, \le k}_n \backslash Mdf_n^{tw, \le k-1}$ where we have used \cite[06FJ,0509]{stacks} to conclude that open and closed substacks behave as expected. 
\end{proof}

Given a twisted modification $\mc{C}'_B$, we define subschemes $p_B,p'_B$ and $p'_b$ for $b \in B_0$  by the fiber product diagrams:
\[
\xymatrix{
p_B\ar[d]\ar[r] & p\ar[d] & &p'_B\ar[d]\ar[r] & p\ar[d] & & p'_b \ar[d]\ar[r] & p_0\ar[d]  \\
C_B\ar[r]\ar[d] & C_S\ar[d]& &\mc{C}'_B\ar[r]\ar[d] & C_S\ar[d] & & \mc{D}'_b \ar[d]\ar[r] & D_0\ar[d]\\
B\ar[r] & S& &B\ar[r] & S & &b\ar[r] & S_0
}
\]
Where $\mc{D}'_B := \mc{C}'_B \x_{C_B} D_B$ and $\mc{D}'_b$ is the restriction to $b \in B$. 

Notice that $p'_b$ is nothing other than the rational chain of $\P^1$s that appear in a modification over the fiber of the node.  Further, $p_B$ and $p'_B$ are defined so that the map $\mc{C}'_B - p'_B \to C_B - p_B$ is an isomorphism.

\subsection{The definition of twisted Gieseker bundles and the completeness statement}\label{s:completeness}
Let $r = rk(G)$; if $\mc{C}'_B$ is a twisted modification of length $\le r$, then a $G$-bundle on $\mc{C}'_B$ is called {\it admissible} if the co-characters determining the equivariant structure at all nodes are linearly independent over $\Q$ and are given by a subset of $\{\eta_0 \dotsc, \eta_r\}$; see \eqref{para.eta} in section \ref{s:notation.para} for the definition of the $\eta_i$.

Let $B$ be an $S$-scheme.  Define a groupoid $\mc{X}_G(C_S)$ over $S$-schemes by the assignment
\[
\mc{X}(C_S)(B) = \left\l \begin{array}{c} \xymatrix@R=.13in{ P_B\ar[d] & \\ \mc{C}'_B\ar[r] &  C_B} \\ \end{array} \right\r
\]
where $\mc{C}'_B$ is a twisted local modification of $C_B$ and $P_B$ is an admissible $G$-bundle on $\mc{C}_B$.  Isomorphisms are commutative diagrams
\[
\xymatrix{
P_B\ar[rr]^{\cong}\ar[d]&&Q_B\ar[d]\\
\mc{C}'_B\ar[dr]\ar[rr]^{\cong}&&\ar[dl]\mc{C}''_B\\
&C_B&
}
\]

For notational convenience we abbreviate $\mc{X}_G(C_S)(B)$ as $\mc{X}_G(B)$.

\begin{thm}\label{thm:isAlgStack}
The functor $\mc{X}_G = \mc{X}_G(C_S)$ is an algebraic stack locally of finite type.  It contains $\mc{M}_G(C_S),\mc{M}_G(C_{S^*})$ as dense open substacks and the complement of $\mc{M}_G(C_{S^*})$ is a divisor with normal crossings.
\end{thm}

\begin{proof}

 We first show that $\mathcal{X}_G$ is a stack fibered in groupoids. Namely, we show (1) for $x, y \in \mathcal{X}_G (B)$ that $U \rightarrow Isom ( x|_U, y|_U)$ is a sheaf on $Sch / B$ and  (2) descent data is effective.

Objects $x, y$ as above consist of $G$-bundles on twisted modifications of order $k$ of some fixed length.  By proposition \ref{p:Mdf}, (1) and (2) holds for twisted local modifications and so it's enough to check (1) and (2) on the additional data of $G$-bundles on a twisted modification.  By definition, $G$-bundles are determined by local gluing data (so (2) holds).  Further, given two $G$-bundles $P,Q$ we can identify the isomorphisms $P \to Q$ as the sections of $P \x Q/G$ over the base and this forms a sheaf so (1) holds.
 
To show $\mc{X}_G$ is algebraic we adapt a proof \cite[Prop.1]{Hein} of Heinloth; namely we will verify Artin's axioms \cite[07Y3]{stacks}. First we recall some deformation theory of $G$-bundles. Let $A$ be a local Artin $\C[[s]]$-algebra with maximal ideal $m$ and residue field $k$.  Let $I \subset A$ be a nilpotent ideal such that $m I = 0$. An object $x \in \mc{X}_G(A/I)$ can be identified with a $G$-bundle $\ol{P}$ on a twisted curve $\mc{C}'_{A/I}$.  If $P$ is an extension of $\ol{P}$ over $A$ then the auomorphisms of $P$ inducing the identity on $\ol{P}$ are classified by $H^0( \mc{C}'_{A/I}, ad(\ol{P})\ox_{A/I} I)$.  The possible extensions are classified by $H^1( \mc{C}'_{A/I}, ad(\ol{P})\ox_{A/I} I)$ and obstructions lie in $H^2( \mc{C}'_{A/I}, ad(\ol{P})\otimes_{A/I} I) = 0$; see \cite{MR2583634} for the case of $GL_r$ and for general $G$ this can be deduced from the proof of \cite[Prop.1]{Hein}.

Artin's axioms can be stated as (1) $\Delta \colon \mc{X}_G \to \mc{X}_G \x \mc{X}_G$ is representable by algebraic spaces, (2)  If $B = \varprojlim B_i$ with $B,B_i$ affine then $\varinjlim \mc{X}_G(B_i) \to \mc{X}_G(B)$ is an equivalence, (3) $\mc{X}_G$ satisfies the Rim-Schlessinger (RS) condition, (4) $H^i( \mc{C}'_{A/I}, ad(\ol{P})\ox_{\C} I)$ for $i = 0,1$ are finite dimensional where $I = (\epsilon) \subset \C[\epsilon]/\epsilon^2 = A$, (5) formal objects come from Noetherian complete local rings $R \supset m$ with $R/m$ finite type over $S$, and (6) $\mc{X}_G$ satisfies openness of versality. We elaborate on (3),(5),(6) when we verify them below.

We also use that any algebraic stack locally of finite type over a locally noetherian base automatically satisfy (1) - (6); see \cite[07SZ]{stacks}.  In particular, the algebraic stack $Mdf^{tw,k}_n$ of proposition \ref{p:Mdf} satisfies (1) - (6). 

By \cite[Cor.3.13]{MR1771927}, we can verify (1) by showing $Isom(x,y) \colon Sch/U \to Sets$ is representable by an algebraic space for every $x,y \in \mc{X}_G(U)$. The objects $x,y$ can be identified with $G$-bundles $P,Q$ over a fixed curve $\mc{C}'_U$.  Then $Isom(x,y)$ can be identified with the sheaf of sections of $P\x_G Q = P \x Q/G$ crossed with $Aut(\mc{C}_U'/C_U)$ which is an algebraic space by \cite[thm 1.1]{MR1432041}.

Statement (2) amounts to showing for any $P \to \mc{C}'_B$ there is an index $j$, a modification $\mc{C}'_{B_j}$ and a $G$-bundle $P_j \to \mc{C}'_{B_j}$ such that $P \to \mc{C}'_B$ is pulled back from $P_j \to \mc{C}'_{B_j}$.  Because twisted local modifications form an algebraic stack we can reduce to showing this for the $G$-bundles.  That is there is a fixed $k$ such that if we define $\mc{C}'_{B_{j+k}}$ as the pull back of $\mc{C}'_{B_j}$ under $B_{j+k} \to B_j$ then $\mc{C}'_B = \varprojlim \mc{C}'_{B_{j+k}}$.  We must then show there is a $j$ such that $P \to \mc{C}'_B$ is pulled back from $P_{j+k} \to \mc{C}'_{B_{j+k}}$ and this follows because $G$-bundles are finitely presented.

For the RS condition suppose we have a pushout $Y' = Y \sqcup_X X'$ with (1)$X,X',Y,Y'$ spectra of local Artin rings of finite type over $S$ and (2) $X \to X'$ a closed immersion.  Then the RS condition states that the functor $\mc{X}_G(Y') \to \mc{X}_G(Y) \x_{\mc{X}_G(X)} \mc{X}_G(X')$ is an equivalence of categories. We show the functor is essentially surjective; that it is fully faithful is a formal argument we omit.  

The condition holds with $\mc{X}_G$ replaced with $Mdf^{tw,k}_n$ so we can assume the following situation
\[
\xymatrix{\mc{C}'_X \ar[r]\ar[d] & \mc{C}'_{X'}\ar[d]\\
\mc{C}'_Y \ar[r] & \mc{C}'_{Y'}}
\]
where all curves are pulled back from $\mc{C}'_{Y'}$.  We further have $G$-bundles $P_X,P_{X'}, P_{Y}$ on the respective curves such that $P_{X'}, P_Y$ extend $P_X$.  We can consider $P_{X'} \in \mc{M}_G(\mc{C}'_{Y'})(X')$ and similarly for $P_X,P_Y$.  The stack $\mc{M}_G(\mc{C}'_{Y'})$ is algebraic by lemma \ref{l:hom}. The latter satisfies the RS condition so there is a $G$-bundle $P_{Y'}$ extending all others and it is necessarily admissible because otherwise the bundles $P_X,P_{X'},P_Y$ would not be admissible.

Statement (4) follows readily because we work with twisted curves which have projective coarse moduli spaces.

A formal object is a triple $\zeta = (R, \zeta_n, f_n)$ where $(R,m)$ is a Noetherian complete ring, $\zeta_n \in \mc{X}_G(\ec R/m^n)$ and $\zeta_{n} \xrightarrow{f_n} \zeta_{n+1}$ are morphisms over $\ec R/m^{n} \to \ec R/m^{n+1}$.  There is a notion of morphisms of formal objects and they form a category. Any $\psi \in \mc{X}_G(R)$ gives rise to a formal object by restriction along $\ec R/m^n \to \ec R$; this is a functor from $\mc{X}_G(R)$  to formal objects over $R$. We must show this is an equivalence.  We show it is essentially surjective; that it is fully faithful follows formally.  

The argument is similar to the verification of (3). Assume now $(R, \zeta_n, f_n)$ is a formal object of $\mc{X}_G$.  Forgetting the data of the $G$-bundle produces a formal object of $Mdf^{tw,k}_l$ where $l$ is the length of modification at the closed point of $\ec R$.   Because $Mdf^{tw,k}_l$ is algebraic, the formal objects comes from a twisted modification $\mc{C}'_R$.  Now the original data of the $G$-bundles on the various $\mc{C}'_{R/m^n} = \mc{C}'_{\ec R} \x_{\ec R} \ec R/m^n$ define a formal object of the algebraic stack $\mc{M}_G(\mc{C}'_R)$ and hence there is a $G$-bundle extending them which, as in the verification of condition (3), is necessarily admissible.

Openness of versality is explained precisely in \cite[07XP]{stacks} but using the Kodoira-Spencer map \cite[2.7]{MR2583634}, as in  \cite[Prop.1]{Hein}, the statement can be simplified.  Let $P_R \to \mc{C}'_R$ be an object of $\mc{X}_G(\ec R)$ and let $P_{univ}$ be the universal bundle over $\mc{C}'_R \x_R \mc{M}_G(\mc{C}'_R)$ and let $\pi$ be the projection to $ \mc{M}_G(\mc{C}'_R)$. Then $P_R$ gives a map $\ec R \xrightarrow{f} \mc{M}_G(\mc{C}'_R)$ and there is an induced Kodoira-Spencer map $\mc{T}_R \to f^*(R^1_{\pi,*} ad(P_{univ}))$ where $\mc{T}_R$ denote the tangent sheaf of $\ec R$. Openness of versality means that this map being surjective is an open condition which follows because the locus where a map of coherent sheaves is surjective is open.  We conclude that $\mc{X}_G(C_S)$ is algebraic.

Let $\eta_i$ be the vertices of $Al$, then $D_i:=\mc{M}_{G,\eta_i}(C_{0,[k]}) \subset \mc{X}_G - \mc{M}_G(C_{S^*})$ and because we have fixed the value of $k$, $D_i$ appears only once in the boundary.  Further, the proof of theorem \ref{Main Theorem} below shows any object $\in \mc{X}_G - \mc{M}_G(C_{S^*})$ is in the closure of some $D_i$ hence $\mc{X}_G - \mc{M}_G(C_{S^*}) = \cup_{i = 0}^r \ol{D_i}$ and thus $\mc{M}_G(C_{S^*})$ is an open sub stack.  Using theorem \ref{thm:Gbundles.chains.para.represented} we conclude that the boundary of the wonderful embedding of the loop group in \cite{Solis} forms an atlas for $\cup_{i = 0}^r \ol{D_i}$ and the former has simple normal crossing singularities.  Finally, $\mc{M}_G(C_S) = \mc{X}_G - \cup_{i \ne 0} \ol{D_i}$, which is open.
\end{proof}

\begin{lemma}\label{l:hom}
Let $\mc{C} \to B$ be a twisted curve over a locally noetherian base $\C$-scheme $B$ and let $H$ be an affine algebraic group over $\C$.  Then the functor $\mc{M}_H(\mc{C}_B)$ which assigns to any $B' \to B$ the groupoid of principal $H$-bundles on $\mc{C}_B \x_B B'$ is an algebraic stack locally of finite type. 
\end{lemma}

\begin{proof}
Writing $pt = \ec \C$, we have $[pt/H] \to pt$ is a morphism of finite presentation hence so is $[pt/H] \x B \to B$.  We observe that $\mc{M}_H(\mc{C}_B) = Hom_B(\mc{C}, [pt/H] \x B)$ and apply \cite{MR2194377} to conclude the result.  Note we must check an additional condition from \cite{MR2258535}; namely that $\mc{M}_H(\mc{C}_B)$ satisfies condition (5) stated in the proof of theorem \ref{thm:isAlgStack}:
\begin{equation}\label{eq:hom}
Hom_R(\mc{C}_R, [pt/H] \x \ec R) \to \varprojlim Hom_{R/m^n}(\mc{C}_{R/m^n},[pt/H] \x \ec R/m^n)
\end{equation}
is an equivalence for any $\ec R \to B$ with $R$ a complete local Noetherian ring $R$.  By \cite{MR2309994}, after an \'etale extension on the base, there is a finite flat morphism $Z \to \mc{C}$ over $B$ with $Z$ a projective scheme.  As in \cite[pg. 50]{MR2194377}, we can verify \eqref{eq:hom} after replacing $\mc{C}$ with $Z$. Then $Hom_B(Z, [pt/H] \x B) = \mc{M}_H(Z_B)$ which is an algebraic stack locally of finite type by \cite{Wang}; in particular \eqref{eq:hom} holds by \cite[07SZ]{stacks}.
\end{proof}

We now come to the main theorem
\begin{thm}\label{Main Theorem}
Let $R = \C[[s]]$ and $K = \C((s))$; for a finite extension $K \to K'$ let $R'$ denote the integral closure of $R$ in $K'$.  Given the right commutative square below, there is finite extension $K\to K'$ and a dotted arrow making the entire diagram commute:
\[
\xymatrix{
\ec K' \ar[r]\ar[d] & \ec K\ar[d]\ar[r]^{h^*} & \mc{X}_G(C_S)\ar[d]\\
\ec R' \ar@{-->}[urr]^{ \ \ \ \ \ \ \ \ \ \ \  \ \ \ \ \ \ \ \ \ \ \ \ \ h}\ar[r] & \ec R\ar[r]^f & S
}
\]
\end{thm}
\begin{proof}
If $f$ factors through $S^* \subset S$ then the morphism $h^*$ determines a $G$-bundle on the smooth curve $C_{S^*}$ and completeness of $\cMG(C_{S^*})$ assures we can extend this to a $G$-bundle over $C_{S^*} \x_{S^*} \ec R'$ which produces the required morphism $h$.

Assume now $f$ is surjective. As in proposition \ref{p:first1/2}, normalize $f$ so it is given by $s \mapsto s^l$ for $l \ge 1$.  Also by proposition \ref{p:first1/2}, the map $h^*$ amounts to a $G$-bundle $P$ on $\C[[x,y,s]]/(x y - s^l) - (0,0,0)$.  By lemma  \ref{l:fixk}, after a finite base change, we can identify $P$ with the restriction of a $G$-bundle on twisted curve of order $k$. Moreover, we can further suppose the equivariant structure of the bundle is determined by a co-character $\eta$ which lies in a face of $Al$.  If $\eta$ happens to be one of the vertices $\eta_i$ of $Al$ then we've determined an objects of $\mc{X}_G(C_S)$ extending $E$.  

In general $\eta$ lies in a higher dimensional face of $Al$ and there is a subset $I = \{\eta_{i_1}, \dotsc, \eta_{i_n}\} \subset \{0, \dotsc, r\}$ such that $P(\eta) = \mc{P}_I$ where $\mc{P}_I$ is defined in \eqref{para.etaI} section \ref{s:notation.para}.  Let $D^{'\frac{1}{k}}_S$ be the iterated blowup of $D^{\frac{1}{k}}_S = \ec \C[[u,v]]$ such that $D^{'\frac{1}{k}}_S \xrightarrow{\pi} \ec \C[[u,v]]$ is a modification of length $\le n-1$ and let $E(\eta_I)$ be the bundle on $[D^{'\frac{1}{k}}_S/\mu_k]$ determined by fixing the equivariant structure at the $j$th node in $\pi^{-1}(0,0)$ to be $\eta_{i_j}$. Because $\eta,\eta_I$ lie in the same face of $Al$ they determine isomorphic bundles hence $E(\eta_I)$ yields an object in $\mc{X}_G(C_S)$ extending $E$.

Finally, if $f$ is the map $s \mapsto 0$ then the map $h^*$ define an admissible $G$ bundles $P_{h^*}$ on a curve $C_{n,[k]}$ as in \eqref{Cnk}.  By definition, there is a subset $I = \{i_j\} \subset \{0, \dotsc, r\}$ of cardinality $n+1$ such that $P_{h^*}|_{D_{n,[k]}}$ is an equivariant bundle with equivariant structure at $p_j \in D_{n,[k]}$ determined by $\eta_{i_j}$.  

After potentially a faithfully flat base change $\ec R' \to \ec R$ the bundle is trivial on the complement of the chain $\cong C_0 - p_0$.  By theorem \ref{thm:Gbundles.chains.para.represented}, fixing a trivialization defines a morphism $\ec K' \to T_{G,I}(D_{n,[k]}) = \frac{\pLG \x \pLG}{Z(L_I) \Delta(L_I) \ltimes (\mc{U}^-_I \x \mc{U}_I)}$. Further we have a commutative diagram
\begin{equation}\label{e:last}
\xymatrix{
\ec K'^*\ar[r]^{\phi \ \ \ \ \ \ \ \  \  }\ar[d] & \frac{\pLG \x \pLG}{Z(L_I) \Delta(L_I) \ltimes (\mc{U}^-_I \x \mc{U}_I)}\ar[d]\\
\ec R' \ar[r]^{\ol{\phi}  \ \ \ \ \ \ \ \  \ \ \ \ \ \ }& \pLG/\mc{P}^-_I \x \pLG/\mc{P}_I .
}
\end{equation}
This follows because $ \pLG/\mc{P}^-_I \x \pLG/\mc{P}_I $ is a projective ind variety.  Let $H = \pLG \x \pLG$, $H_1 = \Delta(L_I) \ltimes (\mc{U}^-_I \x \mc{U}_I)$ and $H_2 = \mc{P}^-_I \x \mc{P}_I$.  Identifying $K' = \C((s))$ and using that $\cup_{m\ge 1} \C((s^{1/m}))$ is the algebraic closure of $\C((s))$ we conclude that after another base change $\ec K'' \to \ec K'$ the element $\phi \in H/H_1((s))$ lifts to an element $\phi' \in H((s^{1/m}))$ for some $m$.  The fact that $\phi$ extends to a map $\ol{\phi}$ on $\C[[s]]$ means that $\phi'$ has a factorization $\phi' = \phi'' \psi$ where $\phi'' \in H[[s^{1/m}]]$ and $\psi \in H_2((s^{1/m}))$.  By applying a change of trivialization over the normalization $\widetilde{D_0}$ we can replace $\phi'$ with $\psi$.  Then using the Levi decomposition of $H_2$ we can factor $\psi = \psi_L \x \psi_U$ where $\psi_L  \in L_I((s^{1/m})) \x L_I((s^{1/m}))$   and $\psi_U \in \mc{U}^-_I((s^{1/m})) \x \mc{U}_I((s^{1/m}))$.  Finally by applying a suitable automorphism over $D_{n,[k]}$ we can replace $\psi$ simply with $\psi_L$.  Altogether the map $\phi$ induces a morphism $\psi_L \colon \C((s^{1/m})) \to H_2 \to H_2/H_1 \cong L_I$.  By abuse of notation let the composition also be denoted $\psi_L$.  Since we have only changed $\phi$ by automorphisms and extensions of the variable, the map $\psi_L$ is in the same isomorphism class of $\phi$.

Using the Bruhat decomposition for loop groups we conclude $\psi_L \in L[[s^{1/m}]] \eta'(s^{1/m}) L[[s^{1/m}]]$, where we again can take $\eta'$ to be in the affine Weyl alcove.  Then as in the previous case we find a subset $I' \subset \{0, \dotsc, r\}$ such that $P(\eta') = P(\eta_{I'})$ where $\eta_{I'} = \sum_{i_j \in I'} \eta_{i_j}$ and use this to construct an object in $\mc{X}_G(C_S)$ extending $h^*$.  

This degeneration terminates when the subset $I = \{0,\dotsc, r\}$ because then the right vertical map in \eqref{e:last} is an isomorphism.
\end{proof}




ÊÊ
\bibliographystyle{plain} 
\bibliography{OGBun}

\end{document}